\documentclass{amsart}
\usepackage{amssymb,amsmath,amscd,graphicx,amsthm}
\usepackage[final]{epsfig}
\usepackage{graphics}

\usepackage{url}
\usepackage{color}

\newtheorem{theorem}{Theorem}
\newtheorem{lemma}[theorem]{Lemma}
\newtheorem{proposition}[theorem]{Proposition}
\newtheorem{corollary}[theorem]{Corollary}

\theoremstyle{definition}

\theoremstyle{remark}
\newtheorem{remark}[theorem]{Remark}

\numberwithin{equation}{section}
\numberwithin{theorem}{section}

\newcommand{\eps}{{\varepsilon}}

\newcommand{\C}{{\mathbf C}}
\newcommand{\N}{{\mathbf N}}

\newcommand{\R}{{\mathbf R}}
\newcommand{\Z}{{\mathbf Z}}
\newcommand{\RP}{{\mathbf {RP}}}
\newcommand{\CP}{{\mathbf {CP}}}

\def\one{\mathbf 1}
\def\bop{{\mathbf p}}
\def\boq{{\mathbf q}}
\def\boxx{{\mathbf x}}
\def\boy{{\mathbf y}}
\def\barT{{\overline{T}}}
\def\barD{{\overline{D}}}
\def\barC{{\overline{C}}}
\def\caq{{\mathcal Q}}
\def\caP{{\mathcal P}}
\def\can{{\mathcal N}}
\def\cac{{\mathcal C}}
\def\caf{{\mathcal R}}
\def\frap{{\mathfrak P}}
\def\diag{\operatorname{diag}}
\def\sgn{\operatorname{sgn}}
\def\gcd{\operatorname{gcd}}
\def\Trace{\operatorname{Tr}}
\def\hcan{{\widehat{\can}}}
\def\EE{{\mathcal E}}
\def\FF{{\mathcal F}}

\def\Poi{{\{\cdot,\cdot\}}}
\newcommand{\PSL}{{\mathrm{PSL}}}
\newcommand{\PGL}{{\mathrm{PGL}}}

\begin{document}
\title[Integrable cluster dynamics and pentagram maps]{Integrable cluster dynamics of directed networks and pentagram maps}
\author[M.~Gekhtman]{Michael Gekhtman}
\address{Department of Mathematics, University of Notre Dame, Notre Dame, IN 46556, USA}
\email{mgekhtma@nd.edu}

\author[M.~Shapiro]{Michael Shapiro}
\address{Department of Mathematics, Michigan State University, East Lansing, MI 48823, USA}
\email{mshapiro@math.msu.edu}

\author[S.~Tabachnikov]{Serge Tabachnikov}
\address{Department of Mathematics,
Pennsylvania State University, University Park, PA 16802, USA
and ICERM, Brown University, Providence, RI 02903, USA}
\email{tabachni@math.psu.edu}

\author[A.~Vainshtein]{Alek Vainshtein}
\address{Department of Mathematics and Department of Computer Science, University of Haifa, Haifa, Mount
Carmel 31905, Israel}
\email{alek@cs.haifa.ac.il}

\begin{abstract}
The pentagram map was introduced by R. Schwartz more than 20 years ago. In 2009, V. Ovsienko, R. Schwartz and S. Tabachnikov
 established Liouville complete integrability of this discrete dynamical system. In 2011,  M. Glick interpreted the pentagram map  as 
a sequence of cluster transformations associated with a special quiver. Using compatible Poisson structures in cluster algebras and 
Poisson geometry of directed networks on surfaces, we generalize Glick's construction to include the pentagram map into a  family of 
discrete integrable maps and we give these maps geometric interpretations.
\end{abstract}

\dedicatory{To the memory of Andrei Zelevinsky}

\maketitle

\tableofcontents

\section{Introduction} \label{intro}

The pentagram map was introduced by R. Schwartz more than 20 years ago \cite{Sch1}. The map acts on plane polygons by drawing the ``
short" diagonals that connect second-nearest vertices of a polygon  and forming a new polygon,  whose vertices are their consecutive 
intersection points, see Fig.~\ref{Penta}. The pentagram map commutes with projective transformations, and therefore acts on the 
projective equivalence classes of polygons in the projective plane.

\begin{figure}[hbtp]
\centering
\includegraphics[height=1.3in]{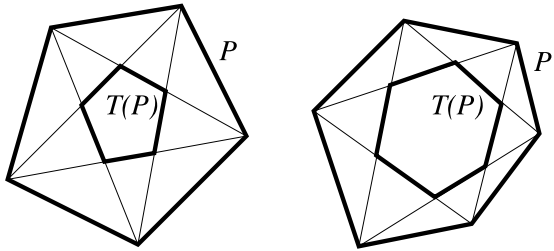}
\caption{Pentagram map}
\label{Penta}
\end{figure}

In fact, the pentagram map acts on a larger class of {\it twisted polygons}. A twisted $n$-gon is an infinite sequence of points $V_i \in \RP^2$  such that $V_{i+n}=M(V_i)$ for all $i \in \Z$ and a fixed projective transformation $M$, called the monodromy. The projective group $\PGL(3,\R)$ naturally acts on twisted polygons. A polygon is closed if the monodromy is the identity. 

Denote by ${\mathcal P}_n$ the moduli  space of projective equivalence classes of twisted $n$-gons, and by ${\mathcal C}_n$ its subspace consisting of closed polygons. Then ${\mathcal P}_n$ and ${\mathcal C}_n$ are  varieties of dimensions $2n$ and $2n-8$, respectively.  Denote by  $T:{\mathcal P}_n\to {\mathcal P}_n$  the pentagram map (the $i$th vertex of the image is the intersection of diagonals $(V_i,V_{i+2})$ and $(V_{i+1},V_{i+3})$.)

One can introduce coordinates $X_1,Y_1,\dots,X_n,Y_n$ in ${\mathcal P}_n$ where $X_i,Y_i$ are the so-called corner invariants associated with $i$th vertex, discrete versions of  projective curvature, see \cite{Sch3}. In these coordinates, the pentagram map is a rational transformation
\begin{equation} \label{pentaformula}
X^*_i=X_i\,\frac{1-X_{i-1}\,Y_{i-1}}{1-X_{i+1}\,Y_{i+1}},
\qquad
Y^*_i=Y_{i+1}\,\frac{1-X_{i+2}\,Y_{i+2}}{1-X_{i}\,Y_{i}}
\end{equation}
(the indices are taken mod $n$). 

In \cite{Sch2}, Schwartz proved that the pentagram map was recurrent, and in \cite{Sch3}, he proved that the pentagram map had $2[n/2]+2$ independent integrals, polynomial in the variables $X_i,Y_i$. He conjectured that the pentagram map was a discrete completely integrable system. 

This was proved in \cite{OST1,OST2}: the space ${\mathcal P}_n$ has a $T$-invariant Poisson structure whose corank equals 2 or 4, according as $n$ is odd or even, and the integrals are in involution. This provides Liouville integrability of the pentagram map on the space of twisted polygons.

F.~Soloviev \cite{So} established algebraic-geometric integrability of the pentagram map by constructing its Lax (zero curvature) representation. His approach established complete integrability of the pentagram map on the space of closed polygons ${\mathcal C}_n$ as well; a different proof of this result was given in \cite{OST3}.

It is worth mentioning that the continuous limit as $n\to \infty$ of the pentagram map is the Boussinesq equation, one of the best known completely integrable PDEs. More specifically, in the limit, a twisted polygon becomes a parametric curve (with monodromy) in the projective plane, and the map becomes a flow on the moduli space of projective equivalence classes of such curve. This flow is identified with the Boussinesq equation, see \cite{Sch1,OST2}. Thus the pentagram map is a discretization, both space- and time-wise, of the Boussinesq equation.

R.~Schwartz and S.~Tabachnikov  discovered several configuration theorems of  projective geometry related to the pentagram map in~\cite{ST1} and   found identities between the integrals of the pentagram map on polygons inscribed into a conic in~\cite{ST2}.
R.~Schwartz~\cite{Sch4} proved that the integrals of the pentagram map do not change in the 1-parameter family of Poncelet polygons (polygons inscribed into a conic and circumscribed about a conic).

It was shown in \cite{Sch3} that the pentagram map was intimately related to the so-called octahedral recurrence (also known as the discrete Hirota equation), and it was conjectured in \cite{OST1,OST2} that the pentagram map was related to cluster transformations. This relation was discovered and explored by Glick \cite{Gl} who proved that the pentagram map, acting on the quotient space ${\mathcal P}_n/\R^*$ (the action of $\R^*$ commutes with the map and is given by the formula
$X_i \mapsto tX_i, Y_i \mapsto t^{-1} Y_i $), is described by coefficient dynamics \cite{FZ4} -- also known as $\tau$-transformations, see Chapter 4 in \cite{GSV3} -- for a certain cluster structure.

In this paper, expanding on the  research announcement \cite{GSTV}, we generalize Glick's work by including the pentagram map into a family of discrete completely integrable systems.
Our main tool is Poisson geometry of weighted  directed networks on surfaces.
The ingredients necessary for complete integrability -- invariant Poisson brackets, integrals of motion in involution, Lax representation -- are recovered from combinatorics of the networks.

A. Postnikov \cite{Po} introduced such networks
in the case of a disk and investigated their transformations and their relation  to cluster
transformations; most of his results are local, and
hence remain valid for networks on any surface. Poisson
properties of weighted directed networks in a disk and their relation to r-matrix
structures on $GL_n$ are studied in \cite{GSV2}. In \cite{GSV4} these results were
further extended to networks in an annulus and r-matrix Poisson
structures on matrix-valued rational functions. Applications of these techniques
to the study of integrable systems can be found in \cite{GSV5}. A detailed
presentation of the theory of weighted directed networks from a cluster algebra
perspective can be found in  Chapters 8--10 of \cite{GSV3}.

Our integrable systems, $T_k$, depend on one discrete parameter $k\ge2$. The geometric meaning of $k-1$ is the dimension of the ambient projective space.  The case $k=3$ corresponds to the pentagram map, acting on planar polygons. 

For $k\ge4$, we interpret $T_k$ as a transformation of a class of twisted polygons in $\RP^{k-1}$, called {\it corrugated polygons}. The map is given by intersecting consecutive diagonals of combinatorial length $k-1$ (i.e., connecting vertex $V_i$ with $V_{i+k-1}$); corrugated polygons are defined as the ones for which such consecutive diagonals are coplanar. The map $T_k$ is closely related with a pentagram-like map in the plane, involving deeper diagonals of polygons.

For $k=2$, we give a different geometric interpretation of our system: the map $T_2$ acts on pairs of twisted polygons in $\RP^1$ having the same monodromy (these polygons may be thought of as ideal polygons in the hyperbolic plane by identifying $\RP^1$ with the circle at infinity), and the action is given by an explicit construction that we call the {\it leapfrog} map  whereby one polygon ``jumps" over another, see a description in Section \ref{geom}. If the ground field is $\C$, we interpret the map $T_2$ in terms of circle patterns  studied by O. Schramm \cite{Sc,BH}.

The pentagram map is coming of age, and we finish this introduction by briefly mentioning, in random order, 
some  related work that appeared since our initial research announcement~\cite{GSTV} was written.
\begin{itemize}
\item A variety of multi-dimensional versions of the pentagram map, integrable and non-integrable, was studied by B. Khesin and F. Soloviev \cite{KhSo1,KhSo2,KhSo3,KhSo4}, and by G. Mari-Beffa \cite{MB,MB2}. The continuous limits of these maps are identified with the Adler-Gelfand-Dikii flows.
\item V. Fock and A. Marshakov \cite{FM} described a class of integrable systems on Poisson submanifolds of the affine Poisson-Lie groups $\PGL(N)$. The pentagram map is a particular example. 
The quotient of the corresponding integrable system by the scaling action (see $p$-dynamics $\bar T_k$ defined in Section~\ref{pq})  coincides with the integrable system constructed by A. Goncharov and R. Kenyon out of dimer
models on a two-dimensional torus and classified by the Newton polygons  \cite{GoK}.
\item M. Glick \cite{Gl2,Gl3} established the singularity confinement property of the pentagram map and some other discrete dynamical systems.
\item R.~Kedem and P.~Vichitkunakorn~\cite{KV} interpreted the pentagram map in terms of $T$-systems.
\item The pentagram map is amenable for tropicalization. A study of the 
tropical limit of the pentagram map was done by T. Kato \cite{Ka}.
\end{itemize}

\section{Generalized Glick's quivers and the\\ $(\bop,\boq)$-dynamics} \label{pq}

For any integer $n\ge 2$, let  $\bop=(p_1,\dots,p_n)$ and $\boq=(q_1,\dots,q_n)$ be independent variables. Fix an
integer $k$, $2\le k\le n$, and consider the quiver (an oriented multigraph without loops and cycles of length 
two) ${\mathcal Q}_{k,n}$
defined as follows: $\caq_{k,n}$ is a bipartite graph on $2n$ vertices labeled $p_1,\dots, p_n$ and $q_1,\dots,q_n$ 
(the labeling is cyclic, so that $n+1$ is the same as $1$). The graph is invariant under the shift $i\mapsto i+1$.
Each vertex has two incoming and two outgoing edges. The number $k$ is the ``span" of the quiver, that is, the distance between two outgoing edges from a $p$-vertex, see Fig.~\ref{quiver} where $r=\lfloor k/2\rfloor-1$ and $r+r'=k-2$ (in other words, $r'=r$ for $k$ even 
and $r'=r+1$ for $k$ odd). For $k=3$, we have Glick's  quiver \cite{Gl}.

\begin{figure}[hbtp]
\centering
\includegraphics[height=1.2in]{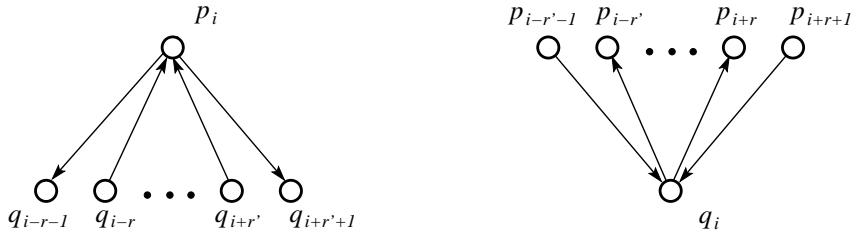}
\caption{The quiver ${\mathcal Q}_{k,n}$}
\label{quiver}
\end{figure}

Let us equip the $(\bop,\boq)$-space with a Poisson structure $\Poi_k$ as follows.  Denote by $A=(a_{ij})$ the $2n\times 2n$ skew-adjacency matrix of  ${\mathcal Q}_{k,n}$, assuming that the first $n$ rows and columns correspond to $p$-vertices. Then we put $\{v_i,v_j\}_k
=a_{ij}v_iv_j$, where $v_i=p_i$ for $1\le i\le n$ and $v_i=q_{i-n}$ for $n+1\le i\le 2n$.

Consider a transformation of the $(\bop,\boq)$-space to itself denoted by $\overline{T}_k$ and defined as follows (the new variables are marked by asterisk):
\begin{equation} \label{mappq}
q_i^*=\frac{1}{p_{i+r-r'}},\quad p_i^*=q_i\frac{(1+p_{i-r'-1})(1+p_{i+r+1})p_{i-r'}p_{i+r}}{(1+p_{i-r'})(1+p_{i+r})}.
\end{equation}

\begin{theorem}\label{pqsyst}
{\rm (i)} The Poisson structure $\Poi_k$ is invariant under the map $\overline{T}_k$.

{\rm (ii)} The function $\prod_{i=1}^n p_iq_i$ is an integral of the map $\overline{T}_k$. Besides, it is
Casimir, and hence the Poisson structure and the map descend to level hypersurfaces of $\prod_{i=1}^n p_iq_i$.
\end{theorem}

\begin{proof}
(i) Recall that given an arbitrary quiver $\caq$, its {\it mutation\/} at vertex $v$ is defined as follows: 

1) for any pair of edges $u'\to v$, $v\to u''$ in $\caq$, the edge $u'\to u''$ is added;

2) all edges incident to $v$ reverse their direction; 

3) all cycles of length two are erased.

\noindent The obtained quiver is said to be {\it mutationally equivalent\/} to the initial one. 
Assume that an independent variable $\tau_v$ is assigned to each vertex of $\caq$. According to Lemma 4.4 of \cite{GSV3}, the cluster transformation of $\tau$-coordinates 
corresponding to the quiver mutation at vertex $v$ (also known as cluster $Y$-dynamics) is defined as follows:
\begin{equation}\label{clust}
\tau^*_v=\frac1{\tau_v},\qquad \tau^*_u=\begin{cases}
\tau_u(1+\tau_v)^{\#(u,v)}\quad\text{if $\#(u,v)>0$},\\
\tau_u\frac{\tau_v^{\#(v,u)}}{(1+\tau_v)^{\#(v,u)}}\qquad\text{if $\#(v,u)>0$},\\
\tau_u\qquad\qquad\qquad\quad\text{otherwise},
\end{cases}
\end{equation}
where $\#(u',u'')$ is the number of edges from $u'$ to $u''$ in $\caq$. Note that at most one of the numbers $\#(u,v)$ and $\#(v,u)$ is nonzero for any vertex $u$.
The cluster structure assosiated with the initial quiver $\caq$ and initial set of variables $\{\tau_v\}_{v\in\caq}$ consists of all quivers mutationally equivalent to $\caq$ and of the corresponding sets 
of variables obtained by repeated application of~\eqref{clust}.

Consider the cluster structure associated with the quiver ${\mathcal Q}_{k,n}$. Choose variables $\bop=(p_1,\dots,p_n)$ and
$\boq=(q_1,\dots,q_n)$ as $\tau$-coordinates,
and consider cluster transformations corresponding to the quiver mutations at the $p$-vertices. These transformations commute, and we perform them simultaneously. By~\eqref{clust}, this leads to the transformation
\begin{equation} \label{exch}
p_i^*=\frac{1}{p_i},\quad q_i^*=q_i\frac{(1+p_{i-r'-1})(1+p_{i+r+1})p_{i-r'}p_{i+r}}{(1+p_{i-r'})(1+p_{i+r})}.
\end{equation}
The resulting  quiver is identical to $\caq_{k,n}$ with the letters $p$ and $q$ interchanged. 
Indeed, the mutation at $p_i$ generates four new edges $q_{i-r}\to q_{i+r'+1}$, $q_{i+r'}\to q_{i+r'+1}$, 
$q_{i-r}\to q_{i-r-1}$, and $q_{i+r'}\to q_{i-r-1}$. The first of them disappears after the mutation at $p_{i+1}$,
the second after the mutation at $p_{i+k-1}$, the third after the mutation at $p_{i-k+1}$, and the fourth after 
the mutation at $p_{i-1}$. Therefore, the result of mutations at all $p$-vertices is just the reversal of all edges of $\caq_{k,n}$.
Thus we compose 
transformation~\eqref{exch} with the 
transformation given by $\bar p_i=q_i$, $\bar q_i =p_{i+r-r'}$  and arrive at the 
transformation $\overline{T}_k$ defined by \eqref{mappq}. The difference in the formulas for the odd and even $k$ is due to the asymmetry between left and right in the enumeration of vertices in Fig.~\ref{quiver} for odd $k$, when $r'\ne r$.

A Poisson structure $\Poi$ is said to be compatible with a cluster structure if  $\{x_i,x_j\}=c_{ij}x_ix_j$ 
for any two variables from the same cluster, where the constants $c_{ij}$ depend on the cluster (the cluster basis is related to the $\tau$-basis described above via monomial transformations; we will not need the explicit description of these transformations here).  
By Theorem~4.5 in \cite{GSV3}, the Poisson structure $\Poi_k$ is compatible with the above cluster structure. Consequently, $\Poi_k$
can be written in the basis~\eqref{exch} in the same way as above via the adjacency matrix of the resulting quiver. After the vertices are renamed, we arrive back at $\caq_{k,n}$, which means that $\Poi_k$ is invariant  under $\overline{T}_k$.

(ii) Invariance of the function $\prod_{i=1}^n p_i q_i$ means the equality $\prod_{i=1}^n p^*_i q^*_i=\prod_{i=1}^n p_i q_i$, which is checked directly by inspection of formulas~\eqref{mappq}. The statement that $\prod_{i=1}^n p_i q_i$ is Casimir (or, equivalently, commutes with any $p_i$ and $q_i$) follows from the form of the quiver, since every vertex has an equal number (2, exactly) of incoming and outgoing edges. Hence, the level hypersurface $\prod_{i=1}^n p_i q_i=\rm{const}$ is a Poisson submanifold, and, moreover, 
$\overline{T}_k$ preserves the hypersurface. 

\end{proof}

Along with the $p$-dynamics $\overline{T}_k$, when the mutations are performed at the $p$-vertices of the quiver 
${\mathcal Q}_{k,n}$, one may consider the respective $q$-dynamics $\overline{T}^\circ_k$, when the mutations are 
performed at $q$-vertices. Let us define an auxiliary map $\barD_k$ given by
\begin{equation}\label{bard}
\bar p_i=\frac1{q_i}, \qquad \bar q_i=\frac1{p_{i+r-r'}}.
\end{equation}
Note that $\barD_k$ is almost an involution: $\barD_k^2=S_{r-r'}$, where $S_t$ is the shift by $t$ in indices.
The following proposition describes relations between transformations $\barT$, $\barT^{-1}$, and $\barT^\circ$.

\begin{proposition}\label{bartinv}
{\rm (i)} Transformation $\barT^\circ_k$ coincides with $\barT_k^{-1}$ and is given by
\begin{equation} \label{mapqp}
p_i^*=\frac{1}{q_{i-r+r'}},\quad q_i^*=p_i\frac{(1+q_{i-r})(1+q_{i+r'})q_{i-r-1}q_{i+r'+1}}{(1+q_{i-r-1})(1+q_{i+r'+1})}.
\end{equation}

{\rm (ii)} Transformations $\barT^\circ_k$ and $\barT_k$ are almost conjugated by $\barD_k$:
\begin{equation}\label{pqconjug}
S_{r-r'}\circ\barT^\circ_k\circ\barD_k=\barD_k\circ\barT_k.
\end{equation}

{\rm (iii)} Let $\barD_{k,n}$ be given by $\bar p_i=q_{i-\lfloor(n+r-r')/2\rfloor}$, $\bar q_i=p_i$. Then
$$
\barT^\circ_k=\barD_{k,n}\circ\barT_{n+2-k}\circ\barD_{k,n}.
$$
\end{proposition}

\begin{proof} (i) Recall that $\barT_k$ is defined as the composition of the cluster transformation~\eqref{exch} and the shift 
$\bar p_i=q_i$, $\bar q_i =p_{i+r-r'}$ Equivalently, we can write $\barT_k=\barC_k\circ\barD_k$, where $\barC_k$ is given by
expressions reciprocal to those in the right-hand side of~\eqref{exch}.  It is easy to check that $\barC_k$ is an involution, and that $\barD_k^{-1}=S_{r'-r}\circ\barD_k$ is given by $\bar p_i=1/q_{i-r+r'}$,
$\bar q_i=1/p_i$. Consequently, $\barT_k^{-1}=\barD_k^{-1}\circ\barC_k$ is given by~\eqref{mapqp}.

To get the same relations for the transformartion $\barT^\circ_k$ one has to use an analog of~\eqref{exch}
$$
q_i^*=\frac1{q_i},\quad p_i^*=p_i\frac{(1+q_{i-r})(1+q_{i+r'})q_{i-r-1}q_{i+r'+1}}{(1+q_{i-r-1})(1+q_{i+r'+1})}
$$
and compose it with the map $\bar q_i=p_i$, $\bar p_i=q_{i-r+r'}$, see Fig.~\ref{quiver}.

(ii) Follows immediately from (i).

(iii) Follows from the fact that $\caq_{n+2-k,n}$ locally at the vertex $q_{i+\lfloor(n+r-r')/2\rfloor}$ has the same structure as $\caq_{k,n}$
at the vertex $p_i$. This is illustrated in Fig.~\ref{pandq}. 

\begin{figure}[hbtp]
\centering
\includegraphics[height=1.2in]{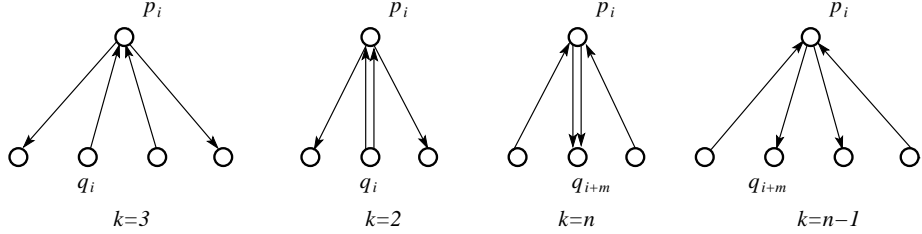}
\caption{The quivers ${\mathcal Q}_{k,n}$ for $n=2m$ and various values of $k$}
\label{pandq}
\end{figure}

For a formal proof note that the values $\hat r$ and $\hat r'$ corresponding to $\hat k=n+2-k$ are given by
\begin{equation}\label{hatr}
\hat r=\lfloor(n+r-r')/2\rfloor-r-1,\qquad \hat r+\hat r'=n-k.
\end{equation}
\end{proof}

By Theorem~\ref{pqsyst}(ii), $\barT_k$ restricts to any hypersurface $\prod_{i=1}^n p_iq_i=c$.
We denote this restriction by $\overline{T}^{(c)}_k$. In what follows, we shall be concerned only with $\overline{T}^{(1)}_k$.
Note that $\overline{T}^{(1)}_3$ is the  pentagram map on ${\mathcal P}_n/\R^*$ considered by Glick \cite{Gl}.

\section{Weighted directed networks and the 
$(\boxx,\boy)$-dynamics} \label{xysect}

\subsection{Weighted directed networks on surfaces}\label{xyint}
We start with a very brief description of the theory of weighted directed networks on surfaces with a boundary, adapted for our purposes; see \cite{Po,GSV3} for details. In this paper, we will only need to consider acyclic graphs on a cylinder (equivalently, annulus) $\mathcal{C}$ that we position horizontally with one boundary circle on the left and another on the right.

 Let $G$ be a directed acyclic graph with the vertex set $V$ and the edge set $E$ embedded in $\mathcal{C}$.
$G$ has $2n$
{\it boundary vertices\/}, each of degree one: $n$ {\it sources\/} on the left boundary circles and $n$ {\it sinks\/} on the right
boundary circle.
All the internal vertices of $G$ have degree~$3$ and are of two types: either they have exactly one
incoming edge (white vertices), or exactly one outgoing edge (black vertices).
To each edge $e\in E$ we assign the {\it edge weight\/} $w_{e}\in\R\setminus 0$.
A {\it perfect network\/} $\can$ is obtained from $G$
by adding an oriented  curve $\rho$ without self-intersections (called a {\it cut\/}) that joins  the left
and the right boundary circles and does not contain vertices of $G$. The points of the {\it space of edge weights\/} $\EE_\can$
can be considered as copies of $\can$  with edges weighted by nonzero real numbers.

 Assign an independent variable $\lambda$ to the cut $\rho$. 
 The weight of a directed path $P$ between a source and a sink
is defined as a signed product of the weights of all edges along the path times  $\lambda^d$, where $d$
is the intersection index of $\rho$ and $P$ (we assume that all intersection points are transversal, in which case the intersection index is the number of intersection points counted with signs). The sign is defined by the rotation number of the loop formed by the path, the cut, and parts of the boundary cycles (see \cite{GSV4} for details). In particular, the sign of a simple path going from one boundary circle to the other one and intersecting the cut $d$ times in the same direction equals $(-1)^d$. Besides, if a path $P$ can be decomposed in a path $P'$ and a simple cycle, then the signs of $P$ and $P'$ are opposite.
The {\it boundary measurement\/} between a given source  and  a given sink is then defined as the sum of path weights over all
(not necessary simple) paths between them. A boundary measurement is rational
 in the weights of edges and $\lambda$, see Proposition~2.2 in~\cite{GSV4}; in particular, if the network
 does not have oriented cycles then 
the boundary measurements are polynomials in edge weights, $\lambda$ and $\lambda^{-1}$.

Boundary measurements are organized in a {\it boundary measurement matrix}, thus giving rise to the {\it boundary measurement map\/}
from $\EE_\can$ to the space of $n\times n$ rational matrix
functions. The gauge group acts on $\EE_\can$ as follows: for any internal vertex $v$ of $\can$ and any Laurent monomial $L$
in the weights $w_e$ of $\can$, the weights of all edges leaving $v$ are multiplied by $L$, and the weights of all edges entering $v$ are
 multiplied by $L^{-1}$. Clearly, the weights of paths between boundary vertices, and hence boundary measurements,
 are preserved under this action.
Therefore, the boundary measurement map can be factorized through the space $\FF_\can$ defined as the
quotient of $\EE_\can$ by the action of the gauge group.

It is explained in \cite{GSV4} that $\FF_\can$ can be parametrized as follows.
The graph $G$ divides $\mathcal C$ into a finite number of connected components called
\emph{faces}. The boundary of each face consists of edges of $G$ and, possibly, of several arcs of
$\partial \mathcal C$.
A face is called {\it bounded\/} if its boundary contains only edges of $G$ and {\it unbounded\/} otherwise.
Given a face $f$, we define its {\it face weight\/}
$y_f=\prod_{e\in\partial f}w_e^{\gamma_e}$,
where $\gamma_e=1$ if the direction of $e$ is compatible with the counterclockwise orientation of the
boundary $\partial f$ and $\gamma_e=-1$ otherwise.
Face weights are invariant under the gauge group action.
Then $\FF_\can$ is parametrized by the collection of all face weights (subject to condition $\prod_f y_f=1$)
 and a weight of an arbitrary path in $G$ (not necessary directed) joining two boundary circles; such a path is called a {\it trail\/}.

Below we will frequently use elementary
transformations of weighted networks that do not change
the boundary measurement matrix. They were introduced by Postnikov in \cite{Po} and are presented in Fig.~\ref{moves}.

\begin{figure}[hbtp]
\centering
\includegraphics[height=2.8in]{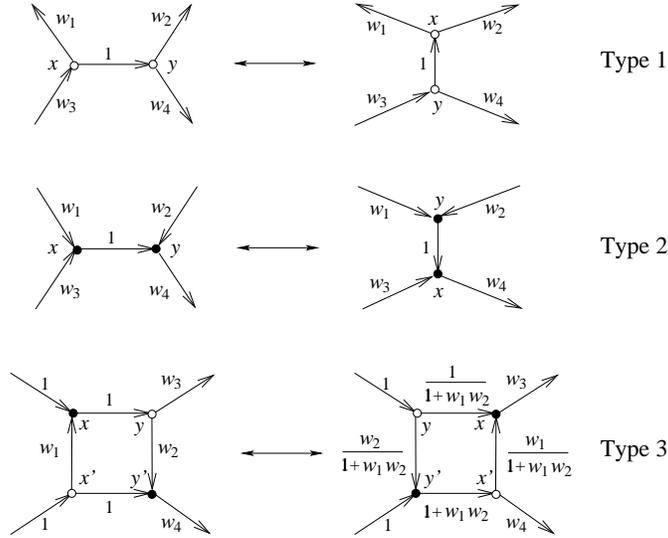}
\caption{Postnikov transformations }
\label{moves}
\end{figure}

Another important transformation is {\it path reversal\/}: for a given closed path one can reverse the directions of all its edges and replace each weight $w_i$ with $1/w_i$. Clearly, path reversal preserves face weights. The transformations of boundary measurements under path reversal are described in \cite{Po, GSV2, GSV4}.

As was shown in \cite{GSV2, GSV4}, the space of edge weights can be made into a Poisson manifold
by considering Poisson brackets
that behave nicely with respect to a natural operation of concatenation of networks. Such Poisson brackets on $\EE_\can$ form a 6-parameter
family, which is pushed forward to a 2-parameter family of Poisson brackets on $\FF_\can$. Here we will need a  specific member of the latter family. The corresponding Poisson structure, called {\em standard}, is described in terms of
the {\it directed dual network\/} $\can^*$ defined as follows. Vertices of $\can^*$ are the faces of $\can$.
Edges of $\can^*$ correspond to the edges of $\can$ that connect either two internal vertices of different colors,
or an internal vertex with a boundary vertex; note that there might be several edges between the same pair of
vertices in $\can^*$. An edge $e^*$ in $\can^*$ corresponding to $e$ in $\can$ is directed in such a way that the white endpoint of $e$ (if it exists) lies to the left of $e^*$ and
the black endpoint of $e$ (if it exists) lies to the right of $e$.
The weight $w^*(e^*)$ equals $1$ if both endpoints of $e$ are internal vertices, and $1/2$ if one of the
endpoints of $e$ is a boundary vertex.  Then
the restriction of the standard Poisson bracket on $\FF_\can$ to the space of face weights is given by
\begin{equation}
\label{facebracket}
\{y_f,y_{f'}\}=\left(\sum_{e^*: f\to f'} w^*(e^*)-
\sum_{e^*: f'\to f} w^*(e^*)\right)y_fy_{f'}.
\end{equation}
The bracket of the trail weight  $z$ and a face weight $y_f$ is given by $\{z, y_f\} = c_f z y_f$. The description of $c_f$ in the general case is rather lengthy. We will only need it in the case when the trail is a directed path $P$ in $G$. In this case
\begin{equation}
\label{trailbracket}
 \{z, y_f\} =\sum_{P'\subset P}\pm\left(\sum_{e\in P', e^*: f\to f'} w^*(e^*)-
\sum_{e\in P', e^*: f'\to f} w^*(e^*)\right)zy_f,
\end{equation}
where each $P'$ is a maximal subpath of $P$ that belongs to $\partial f$, and the sign before the internal sum is positive if $f$ lies to the right of $P'$ and negative otherwise.

Any network $\can$ of the kind described above gives rise to a network $\bar\can$ on a torus. To this end, one identifies boundary circles in such a way that the endpoints of the cut are glued together, and the $i$th source in the clockwise direction from the endpoint of the cut is glued to the $i$th sink in the clockwise direction from the opposite endpoint of the cut. The resulting two-valent vertices are then erased, so that every pair of glued edges becomes
a new edge with the weight equal to the product of two edge-weights involved. Similarly, $n$ pairs
of unbounded faces are glued together into $n$ new faces, whose face-weights are products
of pairs of face-weights involved. We will view two networks on a torus as {\em equivalent} if their
underlying graphs differ only by orientation of edges, but have the same vertex coloring and the same face
weights. The parameter space we associate with $\bar\can$ consists of face weights and the weights $z_\rho$, $z$ of two trails
$P_\rho$ and $P$. 
The first of them is homological to the closed curve on the torus obtained by identifying endpoints of the cut, and the second is noncontractible and not homological to the first one.
The standard Poisson bracket induces a Poisson bracket on face-weights of the new network, which is again given by~\eqref{facebracket} with the dual graph $\can^*$ replaced by $\bar\can^*$ defined by the same rules. The bracket between $z_\rho$ or $z$ and face-weights is given by~\eqref{trailbracket}, provided the corresponding trails are directed paths in $G$.  Finally, under the same restriction on the trails,
\begin{equation}
\label{twotrailbracket}
 \{z, z_\rho\} =\sum_{P'}c_{P'}zz_\rho,
\end{equation}
where each $P'$ is a maximal common subpath of $P$ and $P_\rho$ and $c_{P'}$ is defined in Fig.~\ref{defbrack}.  

\begin{figure}[hbtp]
\centering
\includegraphics[height=2in]{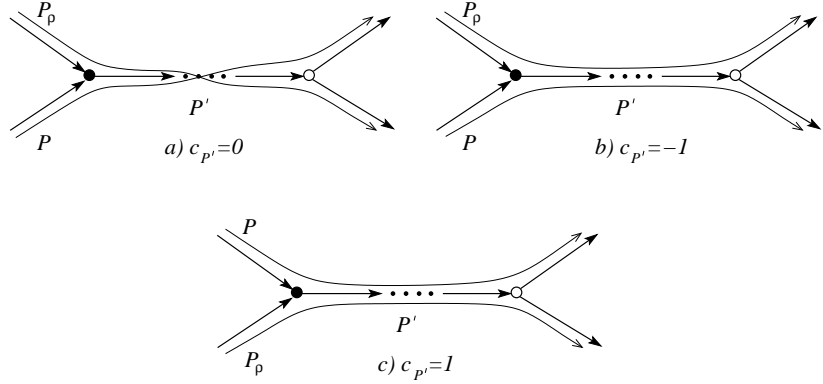}
\caption{To the definition of $c_{P'}$}
\label{defbrack}
\end{figure}

\subsection{The $(\boxx,\boy)$-dynamics}\label{xydyn}
Let us define a network ${\mathcal N}_{k,n}$ on the cylinder. It has $k$ sources, $k$ sinks, and $4n$ internal vertices, of which $2n$
are black, and $2n$ are white. $\can_{k,n}$ is glued of $n$ isomorphic pieces, 
as shown in Fig.~\ref{names}.

\begin{figure}[hbtp]
\centering
\includegraphics[height=1.2in]{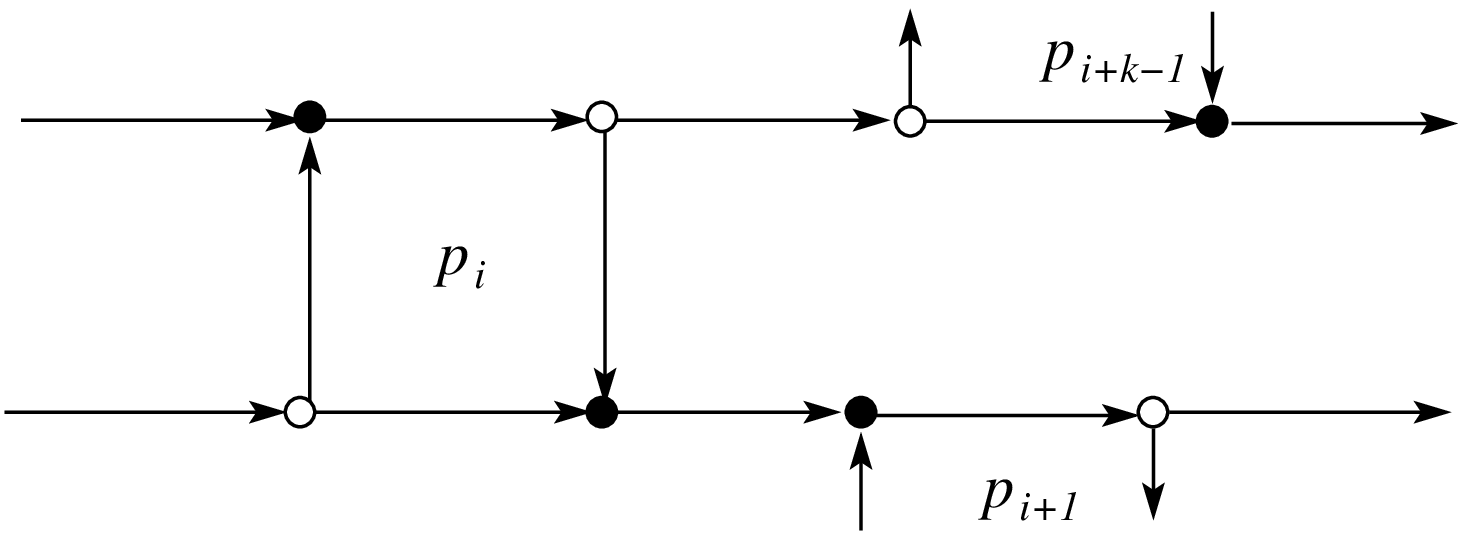}
\caption{Local structure of the networks ${\mathcal N}_{k,n}$  and $\bar\can_{k,n}$}
\label{names}
\end{figure}

The pieces are glued together in such a
way that the lower right edge of the $i$-th piece is identified with the
upper left edge of the $(i+1)$-th piece, provided $i+1\le n$, and the upper right edge of the
$i$-th piece is identified with the lower left edge of the $(i+k-1)$-st
piece, provided $i+k-1\le n$. The network $\bar\can_{k,n}$ on the torus is obtained by dropping the latter
restriction and considering cyclic labeling of pieces. 
The faces of $\bar\can_{k,n}$ are quadrilaterals and octagons. The cut hits only
octagonal faces and intersects each white-white edge. The network $\bar{\mathcal N}_{3,5}$  is shown in Fig.~\ref{network}. The figure depicts a torus, represented as a flattened two-sided  cylinder (the dashed lines are on the ``invisible" side); the edges marked by the same symbol are glued together accordingly.
The cut is shown by the thin line. The meaning of the weights $x_i$ and $y_i$ will be explained later.

\begin{figure}[hbtp]
\centering
\includegraphics[height=1.1in]{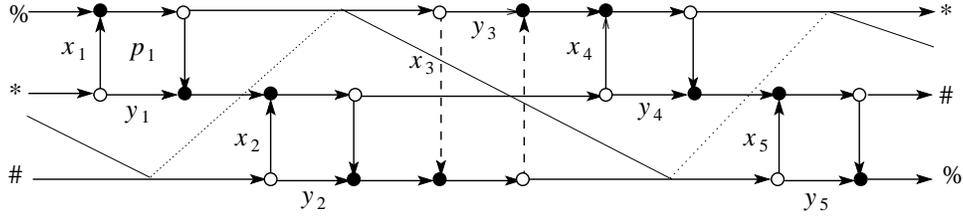}
\caption{The network $\bar{\mathcal N}_{3,5}$ on the torus}
\label{network}
\end{figure}

\begin{proposition}\label{dualnet}
The directed dual of $\bar\can_{k,n}$ is isomorphic to $\caq_{k,n}$.
\end{proposition}

\begin{proof}
It follows from the construction above that $\bar\can_{k,n}$ has $2n$ faces, $n$ of them quadrilaterals and other $n$ octagons.
The quadrilateral faces correspond to $p$-vertices of the directed dual, and octagonal, to its $q$-vertices. Consider the quadrilateral corresponding to $p_i$. The four adjacent octagons are labelled as follows: the one to the left is $q_{i-r-1}$, the one above is $q_{i+r'}$, the one to the right is $q_{i+r'+1}$, and the one below is $q_{i-r}$. Therefore, the octagonal face to the left of the quadrilateral $p_{i+1}$ is $q_{i-r}$, and the one above it is $q_{i+r'+1}$, which justifies the first gluing rule above. 
Similarly, the octagonal face to the left of the quadrilateral $p_{i+k-1}$ is $q_{i+k-2-r}=q_{i+r'}$, and the one below it is
$q_{i+k-1-r}=q_{i+r'+1}$, which justifies the second gluing rule above, see Fig.~\ref{dual}, where the directed dual is shown with
dotted lines. Therefore, we have restored the adjacency structure of $\caq_{k,n}$.
\end{proof}

\begin{figure}[hbtp]
\centering
\includegraphics[height=1.5in]{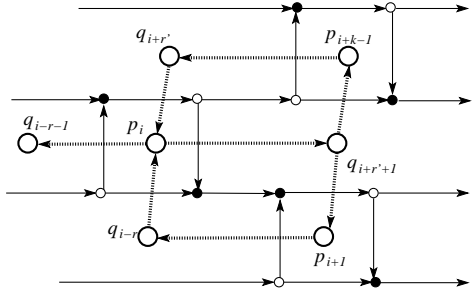}
\caption{The local structure of the directed dual of $\bar\can_{k,n}$}
\label{dual}
\end{figure}

\begin{corollary}\label{twobrackets}
The restriction of the standard Poisson bracket to the space of face weights of $\bar\can_{k,n}$ coincides with
the bracket $\{\cdot,\cdot\}_k$.
\end{corollary}

\begin{proof} Follows immediately from~\eqref{facebracket} and Proposition~\ref{dualnet}, see Fig.~\ref{dual}.
\end{proof}

Assume that the edge weights around the face $p_i$ are $a_i$, $b_i$, $c_i$, and $d_i$, and 
all other weights are equal~$1$, see Fig.~\ref{gauge}. Besides, assume that
\begin{equation}\label{prodbc}
\prod_{i=1}^nb_ic_i=1.
\end{equation}
In what follows we will only deal with weights satisfying the above two conditions.

\begin{figure}[hbtp]
\centering
\includegraphics[height=0.9in]{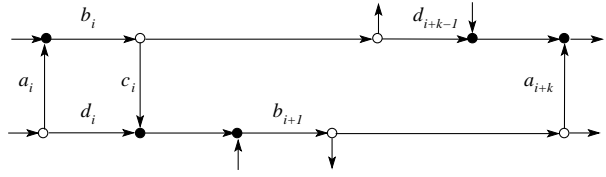}
\caption{Edge weights prior to the gauge group action}
\label{gauge}
\end{figure}

Applying the gauge group action, we can set to~$1$ the weights of the upper and the right edges of each quadrilateral face, while keeping weights of all edges with both endpoints of the same color equal to $1$.
For the face $p_i$, denote by $x_i$ the weight of the left edge and by $y_i$, the weight of the lower edge after the gauge group action (see Fig.~\ref{network}). Put
$\boxx=(x_1,\dots,x_n)$, $\boy=(y_1,\dots,y_n)$.

\begin{proposition}\label{weights} 
{\rm (i)} The weights $(\boxx,\boy)$ are given by
\begin{equation}\label{xyviabc}
x_i=a_ic_{i-k+1}^{-1}\prod_{j=i-k+2}^{i-1}b_j^{-1}c_j^{-1},\qquad y_i=d_ic_{i-k+1}^{-1}\prod_{j=i-k+2}^{i}b_j^{-1}c_j^{-1}.
\end{equation}

{\rm (ii)} The relation between $(\bop,\boq)$ and $(\boxx,\boy)$ is as follows:
\begin{equation}\label{pqviaxy}
p_i=\frac{y_i}{x_i},\quad q_i=\frac{x_{i+r+1}}{y_{i+r}};\qquad x_i=x_1 \prod_{j=1}^{i-1}p_jq_{j-r},\quad y_i=x_i p_i.
\end{equation}
\end{proposition}

\begin{proof}
(i) Assume that the gauge group action is given by $g_i^1$ at the upper left vertex of the $i$th quadrilateral, by $g_i^2$ at the upper right vertex, by $g_i^3$ at the lower left vertex, and by $g_i^4$ at the lower right vertex. The conditions on the upper and right edges of the quadrilateral give $b_ig_i^1/g_i^2=1$ and $c_ig_i^2/g_i^4=1$, while the conditions on the two external edges 
going right from the quadrilateral give $g_i^4=g_{i+1}^1$ and $g_i^2=g_{i+k-1}^3$. Denote $\rho_i=g_i^1/g_i^3$. From the first three
equations above we get $g_{i+1}^3=g_i^3b_ic_i\rho_i/\rho_{i+1}$. Iterating this relation $i+k-1$ times and taking into account the fourth equation above we arrive at $\rho_{i+k-1}=c_i\prod_{j=i+1}^{i+k-2}b_jc_j$, or
$$
\rho_i=c_{i-k+1}\prod_{j=i-k+2}^{i-1}b_jc_j.
$$

Now the first relation in~\eqref{xyviabc} is restored from $x_i=a_ig_i^3/g_i^1=a_i/\rho_i$.
To find $y_i$ we write
$$
y_i=d_i\frac{g_i^3}{g_i^4}=d_i\frac{g_i^3}{g_i^1}\frac{g_i^1}{g_i^2}\frac{g_i^2}{g_i^4}=\frac{d_i}{\rho_ib_ic_i},
$$
which justifies the second relation in~\eqref{xyviabc}. 
Note that $n$-periodicity of $\rho_i$, and hence of $x_i$ and $y_i$, is guaranteed by condition~\eqref{prodbc}. 

(ii) The expression for $p_i$ follows immediately from the definition of face weights. Next, the face weight for the octagonal face to the right of $p_i$ is $q_{i+r'+1}=x_{i+k}/y_{i+k-1}$, which yields $q_i=x_{i+k-r'-1}/x_{i+k-r'-2}=x_{i+r+1}/x_{i+r}$. The
remaining two formulas in~\eqref{pqviaxy} are direct consequences of the first two.
\end{proof}

Note that by~\eqref{pqviaxy}, the projection $\pi_k:(\boxx,\boy)\mapsto (\bop,\boq)$ has a 1-dimensional fiber.
Indeed, multiplying $x$ and $y$ by the same coefficient $t$ does not change the
corresponding $p$ and $q$.

It follows immediately from~\eqref{pqviaxy} that $\prod_{i=1}^n p_iq_i=1$, so the relevant map is 
$\barT^{(1)}_k$. Let us show how it 
can be described via equivalent transformations of the network $\bar\can_{k,n}$.
The transformations include Postnikov's moves of types 1, 2, and 3, and the gauge group action.
 We describe the sequence of these transformations below.

 We start with the network $\bar\can_{k,n}$ with weights $x_i$ and $y_i$ on the left and lower edge of each quadrilateral face.
First, we apply Postnikov's type 3 move at each $p$-face (this corresponds to cluster $\tau$-transformations at $p$-vertices
of $\caq_{k,n}$ given by~\eqref{exch}). To be able to use the type 3 move as shown in Fig.~\ref{moves} we have first to conjugate it
with the gauge action at the lower right vertex, so that $w_1=x_i$, $w_2=1/y_i$, $w_3=1$, $w_4=y$. 
Locally, the result is shown in Fig.~\ref{aftermut} where $\sigma_i=x_i+y_i$.

\begin{figure}[hbtp]
\centering
\includegraphics[height=1in]{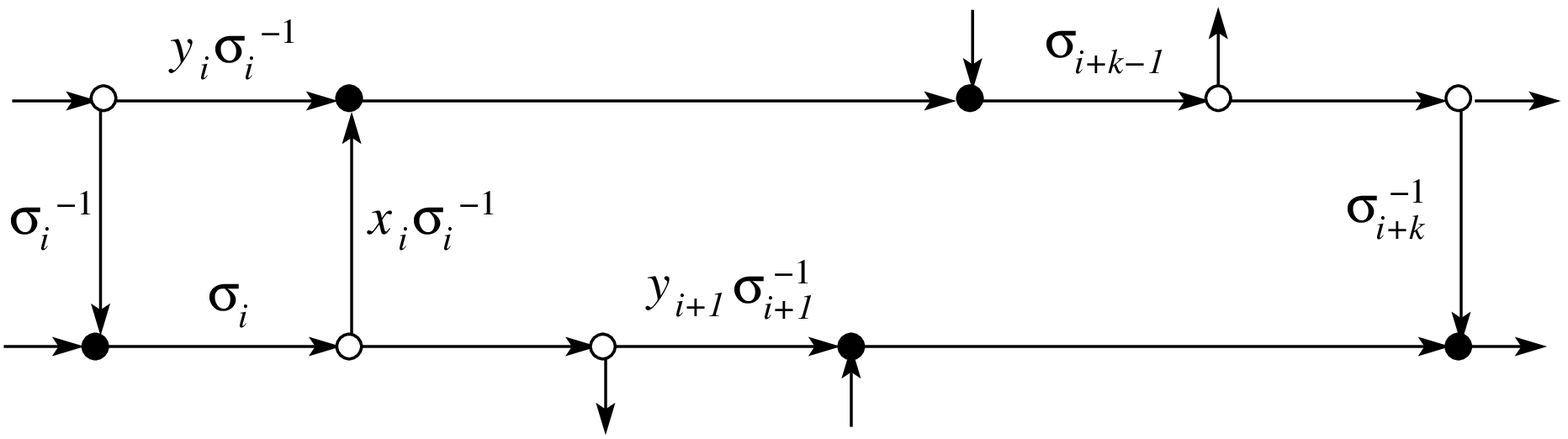}
\caption{Type 3 Postnikov's move for $\bar\can_{k,n}$}
\label{aftermut}
\end{figure}

Next, we apply type 1 and type 2 Postnikov's moves at each white-white and black-black edge, respectively.
 In particular, we move vertical arrows interchanging the right-most and the left-most position on the network in
 Fig.~\ref{network} using the fact that it is drawn on the torus. These moves interchange the quadrilateral and octagonal faces of the graph thereby swapping  the variables $p$ and~$q$, see Fig.~\ref{aftershift}.

\begin{figure}[hbtp]
\centering
\includegraphics[height=1in]{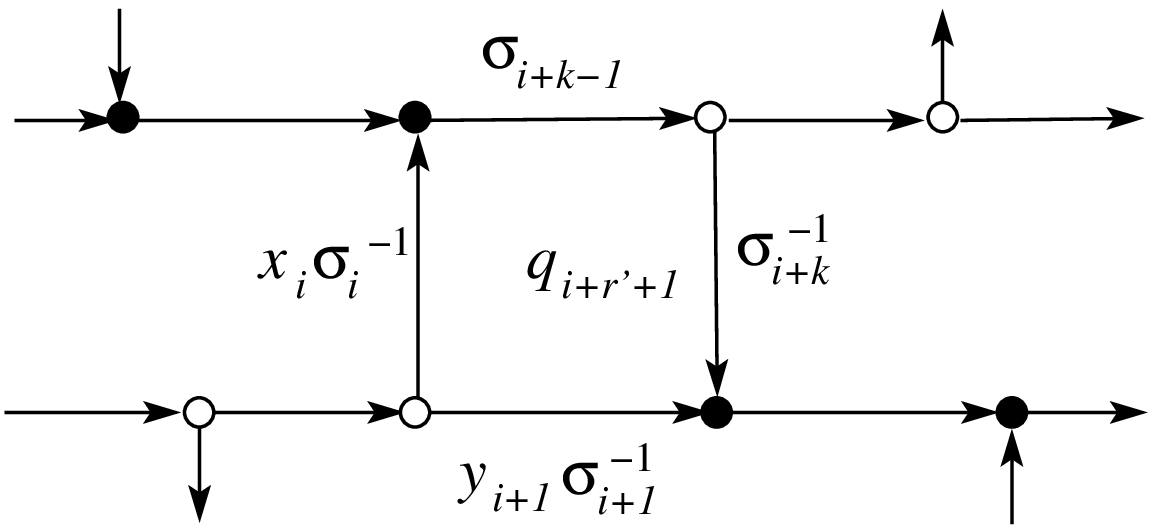}
\caption{Type 1 and 2 Postnikov's moves for $\bar\can_{k,n}$}
\label{aftershift}
\end{figure}

It remains to use gauge transformations to achieve the weights  as in Fig.~\ref{network}. In our situation, weights  $a_i$, $b_i$, $c_i$, $d_i$ are as follows, see Fig.~\ref{aftershift}:
\begin{equation}\label{abcd}
a_i=\frac{x_i}{\sigma_i},\quad b_i=\sigma_{i+k-1},\quad c_i=\frac{1}{\sigma_{i+k}},\quad d_i=\frac{y_{i+1}}{\sigma_{i+1}}.
\end{equation}
Note that condition~\eqref{prodbc} is satisfied.
This yields the map $T_k$, the main character of this paper, described in the following proposition.

\begin{proposition} \label{Tkdef}
{\rm (i)} The map $T_k$ is given by
\begin{equation} \label{mapxy}
x_i^*=x_{i-r'-1} \frac{x_{i+r}+y_{i+r}}{x_{i-r'-1}+y_{i-r'-1}},\quad y_i^*=y_{i-r'} \frac{x_{i+r+1}+y_{i+r+1}}{x_{i-r'}+y_{i-r'}},
\end{equation}

{\rm (ii)} The maps $T_k$ and $\overline{T}^{(1)}_k$ are conjugated via $\pi_k$: $\pi_k\circ T_k=\barT^{(1)}_k\circ \pi_k$.
\end{proposition}

 \begin{proof}
(i) Applying relations~\eqref{xyviabc} to the weights~\eqref{abcd} we get
$$
x^*_{i+r'+1}=x_i\frac{\sigma_{i+k-1}}{\sigma_i},\qquad y^*_{i+r'+1}=y_{i+1}\frac{\sigma_{i+k}}{\sigma_{i+1}},
$$
which immediately implies~\eqref{mapxy}.

(ii) Checked straightforwardly using~\eqref{mappq},~\eqref{pqviaxy}, and~\eqref{mapxy}.
\end{proof}

\begin{remark} 
Note that the map $T_k$ commutes with the scaling action of the group $\R^*$: $(\boxx,\boy)\mapsto (t\boxx,t\boy)$, and that the orbits of this action are the fibers of the projection $\pi_k$.
\end{remark}

Maps $T_2$ and $T_3$ can be further described as follows. The map $T_2$ is a periodic version of the discretization
of the relativistic Toda lattice suggested in \cite{Su}.
It belongs to a family of Darboux-B\"acklund transformations of integrable lattices of Toda type, that were put into a cluster algebras framework in \cite{GSV5}.

\begin{proposition}
The map $T_3$ coincides with the pentagram map. 
\end{proposition}
\begin{proof}
Indeed, for $k=3$,~\eqref{mapxy} gives
\begin{equation} \label{k=3}
x_i^*=x_{i-2} \frac{x_{i}+y_{i}}{x_{i-2}+y_{i-2}},\quad y_i^*=y_{i-1} \frac{x_{i+1}+y_{i+1}}{x_{i-1}+y_{i-1}}.
\end{equation}
Change the variables as follows:
$
x_i\mapsto Y_i,\ y_i\mapsto -Y_iX_{i+1}Y_{i+1}.
$
In the new variables, the map~\eqref{k=3} is rewritten as
$$
X_i^*=X_{i-1} \frac{1-X_{i-2}Y_{i-2}}{1-X_iY_i},\quad Y_i^*=Y_i \frac{1-X_{i+1}Y_{i+1}}{1-X_{i-1}Y_{i-1}},
$$
which becomes formula~\eqref{pentaformula} after the cyclic shift  $X_i\mapsto X_{i+1}, Y_i\mapsto Y_{i+1}$. 
Note that the maps $T_k$, and in particular the pentagram map, commute with this shift.
\end{proof}

Similarly to what was done in the previous section, we may consider, along with the map $T_k$ based on $p$-dynamics $\barT_k$, 
another map, based on $q$-dynamics $\barT^\circ_k$; it is natural to denote this map by $T^\circ_k$. Its definition differs from that
of $T_k$ by the order in which the same steps are performed. First of all, type 1 and 2 Postnikov's moves are applied, which 
leads to quadrilateral faces looking like those in Fig.~\ref{aftermut}. The weights
of the left and the lower edge bounding the face labeled $q_i$ are thus equal to~1, the weight of the upper edge equals $y_{i+r}$, and the weight of the right edge equals $x_{i+r+1}$.  Next, the type 3 Postnikov's move is applied, followed by the gauge group 
action.

An alternative way to describe $T^\circ_k$ is to notice that the network $\bar\can_{k,n}$ can be redrawn in a different way.
 Recall that the network on the torus was obtained from the network on the cylinder by identifying the two boundary circles so that the
cut $\rho$ becomes a closed curve. Conversely, the network on the cylinder is obtained from the network
on the torus by cutting the torus along a closed curve. This curve intersects exactly once $k$ monochrome edges:
the black monochrome edge that points to the face $p_1$ and $k-1$ white monochrome edges that point to the faces $p_1,\dots, p_{k-1}$.
Alternatively, the torus can be cut along a different closed curve that intersects the same black monochrome edge and
all the $n-k+1$ remaining white monochrome edges. An alternative representation of $\bar\can_{3,5}$ is shown in Fig~\ref{alter}. The cut shown in Fig~\ref{alter} coincides  
with that in Fig.~\ref{network}. We can further reverse the closed path shown with
the dashed line and apply type 1 and 2 Postnikov moves at all white-white and black-black edges. It is easy to see that the resulting
network is isomorphic to $\bar\can_{4,5}$. In general, starting with $\bar\can_{k,n}$ and applying the same transformations one
gets a network isomorphic to $\bar\can_{n-k+2,n}$, which hints that $T^\circ_k$ and $T_{n-k+2}$ are related.

\begin{figure}[hbtp]
\centering
\includegraphics[height=3in]{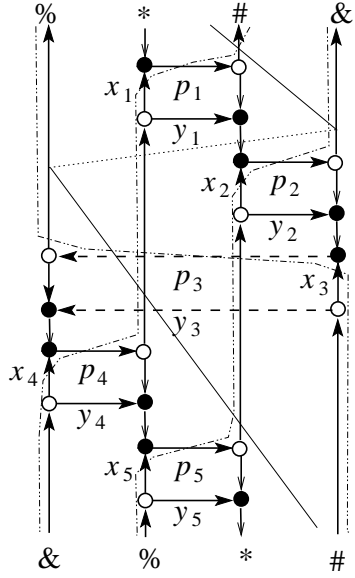}
\caption{An alternative representation of $\bar\can_{3,5}$}
\label{alter}
\end{figure}

Introduce an auxiliary map $D_k$ given by
\begin{equation}\label{mapdk}
x_i^*=\frac1{x_{i+r}}\prod_{j=i-r'}^{i+r-1}\frac{y_j}{x_j}, \qquad
y_i^*=\frac1{x_{i+r+1}}\prod_{j=i-r'}^{i+r}\frac{y_j}{x_j}.
\end{equation}
The following analog of Proposition~\ref{bartinv} explains the relation between $T$, $T^{-1}$ and $T^\circ$.

\begin{proposition} \label{tinv}
{\rm (i)} 
The maps $T_k^{-1}$ and $T^\circ_k$ coincide and are given by
\begin{equation}\label{mapxyinv}
x_i^*=x_{i+r'+1}\frac{x_{i-r}+y_{i-r-1}}{x_{i+r'+1}+y_{i+r'}},\qquad y_i^*=y_{i+r'}\frac{x_{i-r}+y_{i-r-1}}{x_{i+r'+1}+y_{i+r'}}.
\end{equation}

{\rm (ii)} The maps $T_k$ and $T^\circ_k$ are almost conjugated by $D_k$:
\begin{equation}\label{xyconj}
S_{r-r'}\circ T^\circ_k\circ D_k=D_k\circ T_k.
\end{equation}

{\rm (iii)} Let $D_{k,n}$ be given by $\bar x_i=y_{i-r-1}$, $\bar y_i=x_{i-r}$. Then
$$
T^\circ_k=D_{k,n}\circ T_{n-k+2}\circ D_{n-k+2,n}.
$$
\end{proposition}

\begin{proof}
(i)  The proof of~\eqref{mapxyinv} for $T_k^{-1}$ 
is similar to that of Proposition~\ref{bartinv}(i).
It is easy to check that the maps $D_k$ and $\barD_k$ given by~\eqref{mapdk} and~\eqref{bard} are conjugated via $\pi_k$: 
$\pi_k\circ D_k=\barD_k\circ \pi_k$.
Besides, define the map $C_k$ by 
\begin{equation*}
x_i^*=\frac{x_{i-k+1}+y_{i-k+1}}{x_{i-k+1}(x_i+y_i)}\prod_{j=i-k+2}^{i-1}\frac{y_j}{x_j}, \qquad
y_i^*=\frac{x_{i-k+1}+y_{i-k+1}}{x_{i-k+1}(x_i+y_i)}\prod_{j=i-k+2}^{i}\frac{y_j}{x_j}.
\end{equation*}
Similarly, $\pi_k\circ C_k=\barC_k\circ\pi_k$. Moreover, $T_k=C_k\circ D_k$. Therefore, $T_k^{-1}=D_k^{-1}\circ C_k^{-1}$.

A direct computation shows that $C_k$ is an involution, while $D_k^{-1}$ is given by
\begin{equation*}
x_i^*=\frac1{x_{i+r'}}\prod_{j=i-r}^{i+r'-1}\frac{y_j}{x_j}, \qquad
y_i^*=\frac1{x_{i+r'+1}}\prod_{j=i-r}^{i+r'}\frac{y_j}{x_j},
\end{equation*}
and~\eqref{mapxyinv} for $T_k^{-1}$ follows.

To prove~\eqref{mapxyinv} for $T^\circ_k$ one has to perform all the steps described above, similarly to what was
done in the proofs of Propositions~\ref{weights} and~\ref{Tkdef}.

(ii) Follows immediately from (i) and the relation $D_k^2=S_{r-r'}$.

(iii) Checked straightforwardly taking into account~\eqref{hatr}.
Note that transformations $D_{k,n}$ and $D_{n-k+2,n}$ are related to
$\barD_{k,n}$ via $\pi_{n-k+2}\circ D_{n-k+2,n}=\barD_{k,n}\circ \pi_k$ and $\pi_k\circ
 D_{k,n}=\barD_{k,n}\circ\pi_{n-k+2}$.
\end{proof}

\section{Poisson structure and complete integrability} \label{integrability}

The main result of this paper is {\em complete integrability} of
transformations $T_k$, i.e., the existence of a $T_k$-invariant Poisson
bracket and of
a maximal family of integrals in involution. The key ingredient of the proof is the result
obtained
in \cite{GSV4} on Poisson properties of the boundary measurement map
defined in Section~3.1. First, we recall the definition of an R-matrix
(Sklyanin)
bracket, which plays a crucial role in the modern theory of integrable
systems \cite{OPRS, FT}. The bracket is defined on the space of $n\times n$ rational matrix
functions $M(\lambda)=(m_{ij}(\lambda))_{i,j=1}^n$ and is given by
the formula
\begin{equation}
\label{sklya}
\left \{M(\lambda){\stackrel{\textstyle{\small{\otimes}}}{,}} M(\mu)\right\}
= \left [ R(\lambda,\mu), M(\lambda)\otimes M(\mu) \right ],
\end{equation}
where the left-hand is  understood as $\left \{M(\lambda)
{\stackrel{\textstyle{\otimes}}{,}} M(\mu)\right\}_{ii'}^{jj'}=\{m_
{ij}(\lambda),m_{i'j'}(\mu)\}$ and
an R-matrix $R(\lambda,\mu)$ is an operator in $\left(\R^n
\right )^{\otimes 2}$ depending on parameters $\lambda,\mu$ and solving
the classical Yang-Baxter equation. We are interested in the bracket
associated with the {\em trigonometric R-matrix} (for the explicit formula for it, which we will not need,
see \cite{OPRS}).

\subsection{Cuts, rims, and conjugate networks}
Let $\can$ be a perfect network on the cylinder; recall that $\bar\can$ stands for the perfect network  on the torus obtained
from $\can$ via the gluing procedure described in Section~3.1.

\begin{theorem}\label{cyltotor} 
For any perfect network $\hcan$ on the torus, there exists a perfect network $\can$ on the cylinder with sources and 
sinks belonging to different components of the boundary such that
$\bar\can$ is equivalent to $\hcan$, the map $\mathcal{E}_\can \to \mathcal{F}_{\hcan}$ is Poisson with respect to the standard Poisson structures, and spectral invariants of the image $M_\can(\lambda)$ of the boundary measurement map depend 
only on $\mathcal{F}_{\hcan}$. In particular, spectral invariants of  $M_\can(\lambda)$ form an involutive family of functions 
on $\mathcal{F}_{\hcan}$ with respect to the standard Poisson structure.
\end{theorem}

\begin{proof} Consider a closed simple noncontractible oriented loop $\gamma$ on the torus; we call it a {\it rim\/} if it does not pass through vertices of $\hcan$ and its intersection index with the cut $\rho$ equals $\pm1$.
To avoid unnecessary technicalities, we assume that $\gamma$ and all edges of $\hcan$ are smooth curves. Besides, we assume that each edge intersects $\gamma$ in a finite number of points 
and that all the intersections are transversal. Each intersection point defines an orientation of the torus via taking  
the tangent vectors to the edge and to the rim at this point and demanding that they form a right basis of the tangent plane.
We say that the rim is {\it ideal\/} if its intersection points with all edges define the same orientation of the torus.

\begin{proposition} \label{sci}
Let $\hcan$ be a perfect network on the torus, 
then there exists a rim 
which becomes ideal after a finite number of path reversals in $\hcan$.
\end{proposition}

\begin{proof} Consider  the universal covering $\pi$ of the torus by a plane. 
Take an arbitrary rim $\gamma$.
The preimage  $\pi^{-1}(\gamma)$ is a disjoint union
of simple curves in the plane, each one isotopic to a line. Fix arbitrarily one such curve $l_0$; it divides the plane into two regions $L$ and $R$ lying to the left and to the right of the curve, respectively. Let $l_i$, $i\in \N$, be the connected components of $\pi^{-1}(\gamma)$ lying in $R$: $l_1$ is the first one to the right of $l_0$, $l_2$ is the next one, etc. 

Let $\hcan_R$ be the part of the network covering $\hcan$ that belongs to $R$. Each intersection point of an edge of $\hcan$ with $\gamma$ gives rise to a countable number of boundary vertices
of $\hcan_R$ lying on $l_0$. Denote by $m$ the number of intersection points of $\gamma$ with the edges of $\hcan$.
We will need the following auxiliary statement.

\begin{lemma}\label{unbounded}
Let $P$ be a possibly infinite oriented simple path in $\hcan_R$ that ends at a boundary vertex and intersects
$l_{m+1}$. Then there exist $i,j$ such that $m+1\ge i>j\ge0$ and points $t_i\in l_i$ and $t_j\in l_j$ on $P$ such that $t_i$ precedes $t_j$ on $P$ and $\pi(t_i)=\pi(t_j)$.
\end{lemma}

\begin{proof}
 Let us traverse $P$ backwards starting from its endpoint, and let $t_{m+1}$ be the first point on $l_{m+1}$ that is encountered during this process. Further, let $t_i$ for $0\le i\le m$ be the first point on $l_i$ that is encountered while traversing $P$ forward from $t_{m+1}$; in particular, $t_0$ is the endpoint of $P$. The proof now follows from the pigeonhole principle applied to the nested intervals of $P$ between the points $t_{m+1}$ and $t_i$.

\end{proof}

Assume that $\hcan_R$ contains a path $P$ as in Lemma~\ref{unbounded}. Consider the interval $P_{ij}$ of $P$ between the points $t_i$ and $t_j$ described in the lemma. Clearly $\pi(P_{ij})$ is a closed noncontractible path on the torus. If $\pi(P_{ij})$ is a simple path, its reversal increases by one the number of intersection points on $\gamma$ that define a right basis. If $\pi(P_{ij})$ is not simple and $s$ is a point of selfintersection, it can be decomposed into a path from $\pi(t_i)$ to $s$, a loop through $s$, and a path from $s$ to $\pi(t_j)$. Further, the loop can be erased, and the remaining two parts glued together, which results in a closed path on the torus with a smaller number of selfintersection points. After a finite number of such steps we arrive at a simple closed path on the torus that can be reversed.

Proceeding in this way, we get a network $\hcan'$ on the torus equivalent to $\hcan$ such that any path in
$\hcan'_R$ that ends at a boundary vertex does not intersect $l_{m+1}$. Note that a path like that may still be infinite. Each such path divides $R$ into two regions: one of them contains $l_{m+1}$, while the other one is
disjoint from it. Let $A$ be the intersection of the regions containing $l_{m+1}$ over all paths $P$ in $\hcan'_R$, and let $\partial A$ be its boundary. Clearly, $\partial A$ is invariant under the translations that commute with $\pi$ and take each $l_i$ into itself. Therefore, $\pi(\partial A)$ is a simple loop on the torus, and it is homologous to $\gamma$;
it is not a rim yet since it contains edges and vertices of $\hcan'$. 

Each vertex $v$ lying on
$\pi(\partial A)$ has three incident edges. Two of them lie on $\pi(\partial A)$ as well. 
Since a preimage $t$ of $v$ belongs to $\partial A$, the preimages of these two edges incident to $t$ belong
to paths that end at $l_0$. Therefore, if the third edge incident to $v$ is pointed towards $v$, its preimage incident to $t$ should belong to the complement of $A$, by the definition of $A$. 

Now, to build a rim, we take a tubular $\eps$-neighborhood of $\partial A$, and consider the boundary 
$\partial A_{+\eps}$ 
of this tubular neighborhood that lies inside $A$. For $\eps$ small enough, the above property of the vertices 
lying on $\pi(\partial A)$ guarantees that the rim $\pi(\partial A_{+\eps})$ intersects only those edges 
that  point from these vertices into $A$, and hence each intersection point defines a right basis. 
Therefore $\pi(\partial A_{+\eps})$ is an ideal rim. 
\end{proof}

Returning to the proof of the theorem, we 
apply Proposition~\ref{sci}
to find the corresponding ideal rim on the torus. Let $\can$ be the network obtained from $\hcan'$ after we cut the torus along this rim. Note that each edge of $\hcan'$ that intersects the rim yields several (two or more) edges in $\can$; the weights of these edges are chosen arbitrarily subject to the condition that their product equals the weight of the initial edge. By Proposition~\ref{sci}, all sources of $\hcan'$ belong to one of its boundary circles, while all sinks belong to the other boundary circle. Besides, $\bar\can=\hcan'$, and hence $\hcan$ and $\bar\can$ are equivalent. Clearly, one can choose a new cut $\rho'$ on the torus isotopic to $\rho$ such that it intersects the rim only once. Consequently, after the torus is cut into a cylinder, $\rho'$ becomes
a cut on the cylinder.

The rest of the proof relies on two facts. One is Theorem~3.13 of 
\cite{GSV4}: {\em for any network  on a cylinder with the equal number of sources and sinks belonging to 
different components of the  boundary, the standard Poisson structure on the space of edge weights induces the trigonometric R-matrix bracket on the space of boundary measurement matrices}. The second is
a well-known statement in the theory of integrable systems: {\em spectral invariants of $M_{\can}(\lambda)$ are in involution with respect to the Sklyanin bracket\/}, see Theorem~12.24 in \cite{OPRS}.  
\end{proof}

We can now apply Theorem~\ref{cyltotor} to the network $\bar\can_{k,n}$. Clearly, one can choose the rim $\gamma$
in such a way that the resulting network on a cylinder will be $\can_{k,n}$. Note that in this case no
path reversals are needed. For example, for the network $\bar\can_{3,5}$, $\gamma$ can be represented by a closed curve slightly to the left of the edge marked $x_1$ and transversal to all horizontal edges.
The resulting network $\can_{3,5}$ can be seen
in Fig.~\ref{network}, provided we refrain from gluing
together edges marked with the same symbols and regard that figure as representing a cylinder rather than a 
torus. Furthermore, this network on a cylinder is a concatenation of $n$  {\em elementary networks\/} of the same form shown on Fig.~\ref{Lax} (for the cases $k=2$ and $k=3$).

\begin{figure}[hbtp]
\centering
\includegraphics[height=0.9in]{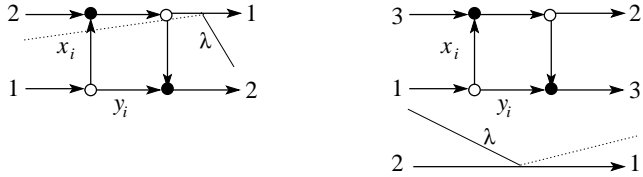}
\caption{Elementary networks}
\label{Lax}
\end{figure}

Since  elementary networks are acyclic, the corresponding  boundary measurement matrices are
$$
L_i(\lambda)=\left(\begin{array}{cc}
-\lambda x_i&x_i+y_i\\
-\lambda &1\\
\end{array}\right)
$$
for $k=2$ and
\begin{equation} \label{Laxmat}
L_i(\lambda)=\left(\begin{array}{cccccc}
0&0&0&\dots&x_i&x_i+y_i\\
-\lambda&0&0&\dots&0&0\\
0&1&0&\dots&0&0\\
0&0&1&\dots&0&0\\
\dots&\dots&\dots&\dots&\dots&\dots\\
0&0&0&\dots&1&1\\
\end{array}\right)
\end{equation}
for $k\ge 3$ (negative signs are implied by the sign conventions mentioned in Section~\ref{xyint}). Consequently, the boundary measurement matrix that corresponds to $\mathcal{N}_{k,n}$ is
\begin{equation}\label{Mmatrix}
M_{k,n}(\lambda)=L_1(\lambda) \cdots L_n(\lambda).
\end{equation}

In our construction above, the cut $\rho$ and the rim $\gamma$ are represented by non-contractible closed curves from two distinct 
homology classes; to get the network $\can_{k,n}$ on the cylinder we start from the network $\bar\can_{k,n}$ on the torus and cut
it along $\gamma$ so that $\rho$ becomes a cut in  $\can_{k,n}$. 
 One can interchange the roles of $\rho$ and $\gamma$ and to cut the torus along $\rho$, making $\gamma$ a cut. This gives another
perfect  network $\can'_{k,n}$ on the cylinder with $n$ sources and $n$ sinks belonging to different components of the boundary. 
To this end, we first observe that $\rho$ intersects all $\circ \to \circ$ edges and no other edges of $\bar\can_{k,n}$ 
(see Fig.~\ref{Lax}). We label the resulting intersection points along $\rho$ by numbers from $1$ to $n$ in such a way that the point
seen on Fig.~\ref{Lax} is labeled by $i$ (for $k\ge3$ this point belongs to the edge that connects the source~$2$ with the sink~$1$).
Next, we cut the torus along $\rho$. Each of the newly labeled intersection points gives rise to one source and one sink in 
$\can'_{k,n}$. The rim $\gamma$ becomes the cut $\rho'$ for $\can'_{k,n}$. It is convenient to view 
$\can'_{k,n}$  as a network in an annulus with sources on the outer boundary circle and sinks on the inner boundary circle. 
The cut $\rho'$ starts at the segment between sinks $n$ and $1$ on the inner circle and ends on the corresponding 
segment on the outer circle. It is convenient to assume that in between it
crosses 
$k-1$ edges incident to the  inner boundary, followed by a single $\bullet \to \bullet$ edge (see Fig.~\ref{conjugate}). 
The variable associated with the cut in $\can'_{k,n}$ will be denoted by $z$. We will say that $\can'_{k,n}$ is {\it conjugate\/} to $\can_{k,n}$.

\begin{figure}[hbtp]
\centering
\includegraphics[height=3.3in]{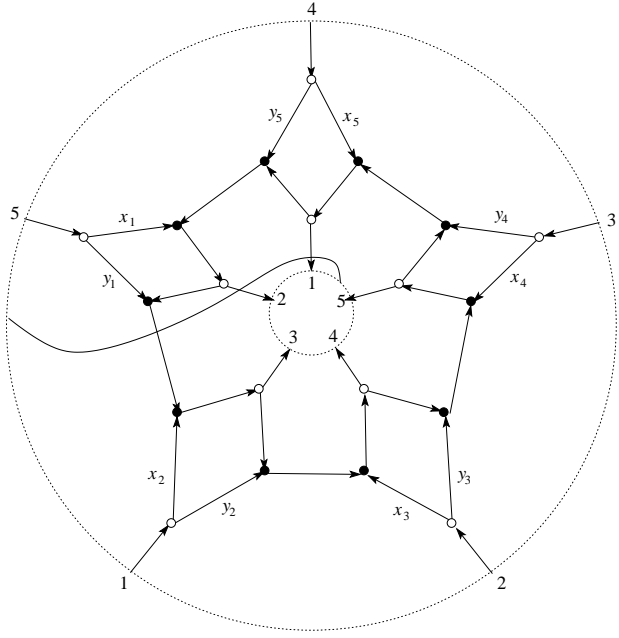}
\caption{The conjugate network $\can'_{3,5}$}
\label{conjugate}
\end{figure}

\begin{proposition} 
\label{conj_bm}
Let $D_x = \diag(x_1,\ldots,x_n)$, $D_y = \diag(y_1,\ldots, y_n)$, and 
$Z = -z e_{n 1} + \sum_{i=1}^{n-1} e_{i,i+1}$. The boundary measurement matrix $A_{k,n}(z)=(a_{ij}(z))_{i,j=1}^n$ 
for the network $\can'_{k,n}$ is given by
\begin{equation}\label{dual_bm}
A_{k,n}(z) = Z\left ( D_x + D_y Z\right ) Z^{k-2} \left ( \one_n - Z\right )^{-1}.
\end{equation}
\end{proposition}

\begin{proof} For any source $i$ and sink $j$ there are exactly two simple (non-self-intersecting) 
directed paths in $\can'_{k,n}$ directed 
from $i$ to $j$ : one contains the edge of weight $x_{i+1}$, and the other, the edge of weight $y_{i+1}$. Every such path has a 
subpath in common with the unique oriented cycle in $\can'_{k,n}$ that contains all edges of weight $1$ that are not incident to either of the boundary components. The weight of this cycle is $z$, which means that all boundary measurements $a_{ij}(z)$ acquire a common factor $1-z+z^2-\ldots=\frac {1} {1+z}$ (see the sign conventions in Section~\ref{xyint}). The simple directed path from $i$ to $j$ containing the
edge of weight $x_{i+1}$ intersects the cut once if $ k-n-1 \leq j-i < k-1 $, twice if $  j-i < k -n-1 $, 
and does not intersect the cut if $j-i \geq k-1$. The simple directed path from $i$ to $j$ containing the
edge of weight $y_{i+1}$ intersects the cut once if $ k-n \leq j-i < k $, twice if $  j-i < k -n $, 
and does not intersect the cut if $j-i \geq k$. All intersections are positive. Thus, by the sign conventions,
$$
(1+z)a_{ij}(z) = \left \{  
\begin{array}{cc} 
x_{i+1} +y_{i+1}   & \text{if} \  j-i > k-1,\\
x_{i+1} -z y_{i+1}   & \text{if} \  j-i = k-1,\\
-z (x_{i+1} + y_{i+1}) & \text{if} \ k-n-1 <  j-i < k-1,\\
-z(x_{i+1} - z y_{i+1})   & \text{if} \  j-i = k-n-1,\\
z^2(x_{i+1} + y_{i+1})  & \text{if} \  j-i < k-n-1.
\end{array}
 \right.
$$
or, equivalently, 
\begin{multline*}
A_{k,n}(z) =\frac {1} {1+z}\\ \times\left ( \diag(x_2,\ldots, x_n, x_1) + \diag(y_2,\ldots, y_n, y_1) Z \right  )  
\left ( Z^{k-1} + \ldots + Z^{n+k-2} \right ).
\end{multline*}
Here we used the relation $Z^n = -z \one_n$. The claim now follows from the identities 
\begin{equation}\label{firstid}
Z\diag(d_1,\ldots, d_n) Z^{-1} = \diag(d_2,\ldots, d_n, d_1)
\end{equation}
and 
\begin{equation}\label{secid}
\left ( \one_n - Z\right )^{-1} =  \frac {1} {1+z} \left ( \one_n + \ldots + Z^{n-1} \right ).
\end{equation}
\end{proof}

\subsection{Poisson structure}
Let $\mathcal M_{k,n}$ and $\mathcal A_{k,n}$ be images of the boundary measurement maps from $\EE_{\can_{k,n}}$ 
and $\EE_{\can'_{k,n}}$ respectively.
Theorem~\ref{cyltotor} implies that
spectral invariants of elements of $\mathcal M_{k,n}$ and $\mathcal A_{k,n}$ viewed as functions on $\FF_{\bar\can_{k,n}}$ are 
in involution with respect to the standard Poisson structure.
However, quantities $x_n, y_n$, and therefore the spectral invariants of $M_{k,n}(\lambda)$ and $A_{k,n}(z)$
are only defined as functions on a subset
$\FF^1_{\bar\can_{k,n}}$ specified by the condition~\eqref{prodbc}.

\begin{proposition}\label{netbracket}
{\rm (i)} $\FF^1_{\bar\can_{k,n}}$ is a Poisson submanifold of $\FF_{\bar\can_{k,n}}$  with respect to 
the standard Poisson structure. For $n\geq 2k-1$, the restriction of the standard Poisson structure to
$\FF^1_{\bar\can_{k,n}}$  is given by
\begin{equation} \label{Poiss}
\begin{aligned}
\{x_i,x_{i+l}\}&=-x_ix_{i+l},\; 1\le l \le k-2; \quad& \{y_i,y_{i+l}\}&=-y_iy_{i+l},\; 1\le l\le k-1;\\
\{y_i,x_{i+l}\}&=-y_ix_{i+l},\; 1\le l\le k-1;\quad& \{y_i,x_{i-l}\}&=y_ix_{i-l},\; 0\le l\le k-2,
\end{aligned}
\end{equation}
where indices are understood $\mod n$ and only non-zero brackets are listed.

{\rm (ii)}  The bracket~\eqref{Poiss} has  rank $2(n-d)$, where  $d =\gcd(k-1, n)$.
Functions  
\begin{equation}
\label{casimirs}
\prod_{i=0}^{\frac{n}{d}-1} x_{s+i (k-1)},\qquad \prod_{i=0}^{\frac{n}{d}-1} y_{s+i (k-1)},\quad s=1,\ldots, d,
\end{equation}
are Casimir functions for~\eqref{Poiss}.

{\rm (iii)} The bracket \eqref{Poiss} is invariant under the map $T_k$.
\label{bracketxy}
\end{proposition}

\begin{proof}
(i) As was explained in Section~\ref{xyint}, $\FF_{\bar\can_{k,n}}$ can be parameterized by face
coordinates $p_i, q_i$, $i=1,\ldots, n$, subject to $\prod_{i=1}^n p_i q_i=1$ and weights $z$, $z_\rho$ of two 
trails that we will choose as follows. The trail that corresponds to $z_\rho$ is a directed cycle $P_\rho$ that 
traces the $\bullet \to \circ \to \bullet$ part of the boundary
of each quadrilateral face $p_i$ and the immediately following $\bullet \to  \bullet$ edge of the corresponding 
octagonal face $q_{i+r'+1}$, see Fig.~\ref{dual}. After applying the gauge action to ensure that weights of all monochrome edges are equal to $1$, we see that the weight $z_\rho$ is equal to $\prod_{i=1}^n b_i c_i$, where we are using notations from Section~\ref{xydyn}. 
The weight $z$ corresponds to the directed cycle $P$ that consists of the $\circ \to \circ$ edge
separating octagonal faces $q_{n-r}$ and $q_{n+1-r}$ followed by the $\circ \to \bullet$ edge of the quadrilateral $p_1$ 
followed by the subpath of $P_\rho$ that closes the cycle. (For example, for $\bar\can_{3,5}$ depicted in 
Fig.~\ref{network}, $P_\rho$ contains the $\circ\to *\to\circ$ edge followed by the $\circ \to \bullet$ edge labeled by $x_1$.)
Since $\FF^1_{\bar\can_{k,n}}$ is cut out from $\FF_{\bar\can_{k,n}}$ by condition~\eqref{prodbc}
(or, equivalently, $z_\rho=1$), to see that  $\FF^1_{\bar\can_{k,n}}$ is a Poisson submanifold of $\FF_{\bar\can_{k,n}}$, we need to 
check that Poisson brackets of $z_\rho$ with $z$ and all face weights with respect to the standard Poisson 
structure are zero. For the bracket $\{z_\rho,z\}$ this claim follows from~\eqref{twotrailbracket}: there is only one maximal common subpath of $P$ and $P_\rho$, and the relative position of the paths is as on Fig.~\ref{defbrack}a). For the bracket $\{z_\rho,y_f\}$ the claim follows from~\eqref{trailbracket}. If $f$ is
a quadrilateral face then there is only one path $P'$ in the outer sum; it consists of two edges, and the corresponding edges of the directed dual are opposite, see Fig.~\ref{dual}. If $f$ is an octagonal face then there
are two paths $P'$ in the outer sum. One of them consists of three edges, of which one is monochrome; the edges 
of the directed dual corresponding to the remaining two edges of $P'$ are opposite, see Fig.~\ref{dual}. The second one consists of a unique monochrome edge. 

Next, $\FF^1_{\bar\can_{k,n}}$ can be parameterized by either $x_i, y_i$, $i=1,\ldots, n$, or by
 $p_i, q_i$, $i=1,\ldots, n$, $\prod_{i=1}^n p_i q_i=1$,  and $z$. To finish the proof of statement (i), it suffices to show that brackets~\eqref{Poiss} generate the same Poisson relations among $p_i, q_i, z$ as the standard Poisson structure on $\FF_{\bar\can_{k,n}}$. Recall that by Corollary~\ref{twobrackets}, 
 the Poisson brackets between $p_i, q_i$ in the standard Poisson structure coincide with those given by 
 $\{\cdot ,\cdot \}_k$. 
 Furthermore, it follows from~\eqref{trailbracket} that $\{ z, q_i\} = 0$ and 
 $\{ z, p_j\} = \left ( \delta_{1,j} - \delta_{n-k+2, j}\right ) z p_j $.
  Note that due to~\eqref{prodbc} and gauge-invariance of weights of directed cycles, $z=x_1$ on $\FF^1_{\bar\can_{k,n}}$. This, together with the periodicity of $\bar\can_{k,n}$, leads to Poisson brackets  
  $\{ x_i, p_j\} = \left ( \delta_{i,j} - \delta_{i-k+1, j}\right ) x_i p_j$.

For an $n$-tuple $(u_1,....u_n)$, let  $\bar {\bf u}$ be the column vector $(\log u_i)_{i=1}^n$. For two $n$-tuples $(u_1,....u_n)$, $(v_1,....v_n)$ of functions on a Poisson manifold, we use a shorthand notation 
$\{ \bar {\bf u}, \bar {\bf v}^T\}$ to denote a matrix of Poisson brackets 
$\left ( \{\log u_i, \log v_j \}\right )_{i,j=1}^n$. 
Note that $\{ \bar {\bf v}, \bar {\bf u}^T\}=-\{ \bar {\bf u}, \bar {\bf v}^T\}^T$.

 We can then describe the Poisson brackets $\{p_i, q_j\}$, $\{p_i, p_j\}$, $\{q_i, q_j\}$, $\{x_i, p_j\}$  by
\begin{equation}\label{bar_pq}
\begin{split}
\{ \bar {\bf p}, \bar {\bf q}^T\} = C^{-r-1} + C^{ r'+1} - C^{-r} - C^{ r'},\\ 
\{ \bar {\bf p}, \bar {\bf p}^T\}=\{ \bar {\bf q},\bar {\bf q}^T\} = \{ \bar {\bf x}, \bar {\bf q}^T\} =0,\quad
 \{ \bar {\bf x}, \bar {\bf p}^T\} = \one - C^{1-k},
\end{split}
\end{equation}
where $C= e_{12} + \cdots + e_{n-1 n} + e_{n 1} = S + e_{n 1}$ is an $n\times n$ cyclic shift matrix 
and $S$ is an upper triangular shift matrix. Similarly, formulas in~\eqref{Poiss}  are equivalent, 
provided $n\ge 2k-1$, to 
\begin{align}
\nonumber
\Omega_x&:=
\{ \bar {\bf x}, \bar {\bf x}^T\} = \sum_{i=1}^{k-2} \left ( C^{-i} - C^{i}\right ) = (\one - C^{k-1})\sum_{i=1}^{k-2}  C^{-i},\\
\label{omegas}
\Omega_y&:=
\{ \bar {\bf y}, \bar {\bf y}^T\} = \sum_{i=1}^{k-1} \left ( C^{-i} - C^{i}\right )=(\one - C^{k-1})\sum_{i=0}^{k-1}  C^{-i},\\
\nonumber
\quad \Omega_{yx}&:=
\{ \bar {\bf y}, \bar {\bf x}^T\} = \sum_{i=1}^{k-1} \left ( C^{1-i} - C^{i}\right )=(\one - C^{k-1})\sum_{i=0}^{k-2}  C^{-i}.
\end{align}
We need to check that relations~\eqref{omegas} imply~\eqref{bar_pq}. 
This follows via a straightforward calculation
from relations $\bar {\bf p} = \bar {\bf y} - \bar {\bf x}$, $\bar {\bf q} = C^r \left ( C \bar {\bf x} - \bar {\bf y}\right )$ induced by~\eqref{pqviaxy} (one also needs to take into account equalities $r+r'=k-2$ and $C^T=C^{-1}$.)

(ii) The rank of the Poisson bracket~\eqref{Poiss} is equal to the rank of the matrix
$$
\Omega=\left (\begin{array}{cc} \Omega_x & -\Omega_{yx}^T\\   
\Omega_{yx} & \Omega_{y}\end{array}\right ).
$$ 
The claim that functions~\eqref{casimirs} are Casimir functions follows from~\eqref{omegas} and the
fact that vectors $\sum_{i=0}^{\frac{n}{d}-1} e_{s+i (k-1)}$, $s=1,\ldots, d$, form a basis of the kernel
of $\one - C^{k-1}$. 
Let $V$ be the complement to that kernel in ${\R}^n$ spanned by vectors $(v_i)_{i=1}^n$ such that  $\sum_{i=0}^{\frac{n}{d}-1} v_{s+i (k-1)}=0$ for $s=1,\ldots, d$.
Then $V$ is invariant under $C$, $\one - C^{k-1}$ is invertible on  $V$ and the rank of $\Omega$ is equal to the rank of its restriction to  $V \oplus V$. On $V \oplus V$, we define 
$$
A=\left (\begin{array}{cc} C (C-\one)^{-1}  & - (C-\one)^{-1} \\   
-\one & \one \end{array}\right )
$$ 
and compute 
$$ 
A \Omega A^T = \left (\begin{array}{cc} 0 & C^{1-k} - \one  \\   
\one - C^{k-1} & 0  \end{array}\right ).
$$
Since $ A \Omega A^T$ is invertible on $V\oplus V$, we conclude that the rank of~\eqref{Poiss} is
$ 2 (n- d)$.

(iii) Invariance of \eqref{Poiss} under the map $T_k$ can be verified by a direct calculation.
\end{proof}

\begin{remark}
There are formulae similar to~\eqref{Poiss} for  $T_k$-invariant Poisson bracket in the case
 $n< 2k-1$ as well. Our focus on the ``stable range'' $n\geq 2k-1$ will be justified by the geometric interpretation of the maps $T_k$ in Section~\ref{geom}.
\end{remark}

\subsection{Conserved quantities}
 The ring of spectral invariants of $M_{k,n}(\lambda)$ is generated by coefficients of its characteristic polynomial
\begin{equation} 
\det(I_n+zM_{k,n}(\lambda)) = \sum_{i=1}^n \sum_{j=1}^k I_{ij}(x,y)\lambda^i z^j.
\label{integrals}
\end{equation}
(Some of the coefficients $I_{ij}$ are identically zero.)

\begin{proposition}
\label{invariants}
Functions $I_{ij}(x,y)$ are invariant under the map $T_k$.
\end{proposition}

\begin{proof}
Recall that in Section~\ref{xydyn}, $T_k$ was described via a sequence of Postnikov's moves and gauge transformations. 
Furthermore,  $\mathcal{N}_{k,n}$ is obtained from  $\bar\can_{k,n}$
by cutting the torus into a cylinder along an ideal rim $\gamma$. 
Note that type 3 Postnikov's moves and gauge transformations do not affect the boundary measurement matrix. In fact, the only transformations that do change the boundary measurements are type 1 and 2 moves interchanging vertical edges lying on different sides of $\gamma$. For a network on a cylinder, moving a vertical edge past $\gamma$ from left to right is equivalent to cutting at the right end of the cylinder a thin cylindrical slice  containing this edge and no other vertical edges and then reattaching this slice to the cylinder on the left. In terms of boundary measurement matrices, this operation amounts to a matrix transformation of the form $ M=A B \mapsto \tilde M = B A$ under which non-zero eigenvalues of $M$ and $\tilde M$ coincide. This proves the claim.
\end{proof}

Next, we will provide a combinatorial interpretation of conserved quantities $I_{ij}$ in terms of the network $\bar\can_{k,n}$. This, in turn, will allow us to clarify
the relation between boundary measurements $M_{k,n}(\lambda)$ and $A_{k,n}(z)$ in the context of the map $T_k$.

Let $\can$ be a perfect network on the torus with the cut $\rho$, and let $\gamma$ be a rim.
For an arbitrary simple directed cycle $C$ in $\can$ we define its weight $w(C)$ as the product of the weights of the edges in $C$ times $(-1)^{d_\lambda+d_z}\lambda^{d_\lambda}z^{d_z}$, where $d_\lambda$ and $d_z$ are the intersection indices of $C$ 
with $\rho$ and $\gamma$, respectively. The weight of a collection $\cac$ of disjoint simple cycles is defined as
$w(\cac)=(-1)^{|\cac|}\prod_{C\in\cac}w(C)$. Finally, define the function 
$\frap_\can(\lambda,z)=\sum w(\cac)$, where the sum is taken over all collections of disjoint simple cycles.

\begin{proposition}\label{talaska}
Let $\can$ be a perfect network on the torus with no contractible cycles, $\gamma$ be an ideal rim, and $M(\lambda)$
be the $m\times m$ boundary measurement matrix for the network on the cylinder obtained by cutting the torus along $\gamma$.
Then
\begin{equation}\label{charviapar}
\det(I_m+zM(\lambda))=\frac {\frap_\can(\lambda,z)}{\frap_\can(\lambda,0)}.
\end{equation}
\end{proposition}

\begin{proof} First of all, note that
\begin{equation}\label{charpol}
\det(I_m+zM(\lambda))=1+\sum_{j=1}^m z^j\sum_{|J|=j}\Delta^J(M(\lambda)),
\end{equation}
where $\Delta^J(M(\lambda))$ is the principal $j\times j$ minor of $M(\lambda)$ with the row and column sets $J$. To evaluate
this minor we use the formula for determinants of weighted path matrices obtained in~\cite{Ta}. 

It is important to note that there are two distinctions between the definitions of the path weights here and in~\cite{Ta}.
First, there is no cut in~\cite{Ta}. This can be overcome by modifying edge weights: if an edge of weight $w$  intersects the cut, then its weight is changed to $\lambda w$ or $\lambda^{-1}w$, depending on the orientation of the intersection; see Chapter~9.1.1 in~\cite{GSV3} for details. Second, the
sign conventions in~\cite{Ta} are different from those described in Section~\ref{xyint}: the sign of any path is positive.
However, in the absence of contractible cycles our conventions can emulate conventions of~\cite{Ta}. To achieve that, it suffices
to apply the transformation $\lambda\mapsto -\lambda$. Indeed, after the torus is cut along $\gamma$, the only cycles in $\can$ that survive are those with $d_\lambda=\pm1$. By our sign conventions, such cycles contribute $-1$ to the sign of a path. The same result
is achieved if the contribution to the sign of a path is $1$, and the weight of the appropriate edge is multiplied (or divided) 
by $-\lambda$. Paths on the  cylinder that intersect the cut $\rho$ are treated in a similar way. Finally, the rim $\gamma$ is ideal,
 and hence the sources and the sinks lie on different boundary circles of the cylinder.
Therefore, 
\begin{equation}\label{postal}
\Delta^J(M(\lambda))=\Delta^{J,m+J}(\bar W(-\lambda)), 
\end{equation}
where $\bar W(-\lambda)$ is the path weight matrix built by the rules of~\cite{Ta} based on modified weights 
of the edges.

It follows from the main theorem of~\cite{Ta} that 
\begin{equation}\label{talexp}
\Delta^{J,m+J}(\bar W)=\frac{\sum_\caf \sgn(\caf) \bar w(\caf)}{\sum_\cac (-1)^{|\cac|} \bar w(\cac)}
\end{equation}
where $\caf$ runs over all collections of $j$ disjoint paths connecting sources from $J$ with the sinks from $m+J$, 
$\cac$ runs over all collections of disjoint cycles in the network on the cylinder, and $\sgn(\caf)$ is the sign of the permutation 
$\pi_\caf$ of size $j$ realized by the paths from the collection $\caf$. 
Equation~\eqref{talexp} takes into account that there are no contractible cycles in  $\can$, and hence any cycle 
that survives on the cylinder intersects any path between a source and a sink. Consequently, the denominator 
in~\eqref{talexp} equals $\frap_\can(\lambda,0)$.

To proceed with the numerator, assume that $\pi_\caf$ can be 
written as the product of $c$ cycles of lengths $l_1,\dots,l_c$ subject to $l_1+\cdots+l_c=j$. Then 
$\sgn(\pi_\caf)=(-1)^{l_1-1}\cdots(-1)^{l_c-1}=(-1)^{j-c}$. It is easy to see that on the torus, the paths from 
$\caf$ form exactly $c$ disjoint cycles, and that the intersection index of the $i$th cycle with $\gamma$ 
equals $l_i$. By~\eqref{charpol}--\eqref{talexp}, we can write
\begin{equation*}
\det(I_m+zM(\lambda))=\dfrac{\sum_{j=0}^m z^j\sum_{|J|=j}\sum_\cac (-1)^j(-1)^{|\cac|}\prod_{C\in\cac}
\bar w(C)}{\frap(\lambda,0)},
\end{equation*}
where the inner sum is taken over all collections $\cac$ that intersect $\gamma$ at the prescribed set $J$ of 
points. Clearly, the numerator of the above expression equals $\frap_\can(\lambda,z)$, and~\eqref{charviapar} follows. 
\end{proof}

\begin{corollary}\label{twocharpol}
One has
\begin{equation}\label{tcp}
\det(I_k+zM_{k,n}(\lambda))=(1+z)\det(I_n+\lambda A_{k,n}(z))=\frap_{\bar\can_{k,n}}(\lambda,z).
\end{equation}
\end{corollary}

\begin{proof}
It is easy to see that the network $\bar\can_{k,n}$ does not have contractible cycles. Besides, both $\gamma$ 
and $\rho$ are ideal rims with respect to each other. Therefore, by Proposition~\ref{talaska}, 
\begin{equation*}
\det(I_k+zM_{k,n}(\lambda))=\frac{\frap_{\bar\can_{k,n}}(\lambda,z)}{\frap_{\bar\can_{k,n}}(\lambda,0)}, \qquad
\det(I_n+\lambda A_{k,n}(z))=\frac{\frap_{\bar\can_{k,n}}(\lambda,z)}{\frap_{\bar\can_{k,n}}(0,z)}.
\end{equation*}

Next, the network $\can_{k.n}$ is acyclic, and so $\frap_{\bar\can_{k,n}}(\lambda,0)=1$. 
Finally, the weight of the only simple cycle in the conjugate network $\can'_{k,n}$ equals $z$,
hence $\frap_{\bar\can_{k,n}}(0,z)=1+z$ and~\eqref{tcp} follows.
\end{proof}

\subsection{Lax representations}
Another way to see the invariance of $I_{ij}$ under $T_k$ is  based on a {\em zero curvature Lax representation\/} with a spectral parameter. A zero curvature representation for a nonlinear dynamical system is a compatibility condition for an over-determined system of linear equations; 
this is a powerful method of establishing algebraic-geometric complete integrability, see, e.g., \cite{DKN}. Even more generally, the term ``Lax representation'' is often used for discrete systems that can be described via a re-factorization of matrix rational functions $ A(z) = A_1(z) A_2(z) \mapsto A^*(z) = A_2(z) A_1(z)$, see, e.g., \cite{MV}. 

\begin{proposition}
\label{zero-curv}
The map $T_k$ has a $k\times k$ zero curvature representation 
$$
L_i^*(\lambda)= P_i(\lambda) L_{i+r-1}(\lambda) P_{i+1}^{-1}(\lambda)
$$
and an $n\times n$ Lax representation
$$
A_{k,n}(z) = A_1(z) A_2(z)\ \mapsto\ A^*_{k,n}(z) =  A_2(z) A_1(z).
$$
Here the Lax matrices $L_i(\lambda)$ and $A_{k,n}(z)$ are defined 
 by~\eqref{Laxmat} and~\eqref{dual_bm}, respectively, and  
 $L_i^*(\lambda)$  and $A^*_{k,n}(z)$ are their images under the transformation $T_k$. 
The  auxiliary  matrix $P_i(\lambda)$ is given by
$$
P_i(\lambda) = \left ( 
\begin{array} {cc}
-\frac {x_{i-1}}{\sigma_{i-1}} - \frac{1} {\lambda \sigma_{i}} & \frac{1}{\lambda}\\
-\frac{1} { \sigma_{i}} & 0
\end{array}
\right )
$$
for $k=2$ and 
$$
P_i(\lambda)=\left(\begin{array}{ccccccc}
0&-\frac{x_i}{\lambda \sigma_i}&-\frac{y_{i+1}}{\lambda \sigma_{i+1}}&0&\dots&0&0\\
0&0&\frac{x_{i+1}}{\sigma_{i+1}}&\frac{y_{i+2}}{\sigma_{i+2}}&\dots&0&0\\
\dots&\dots&\dots&\dots&\dots&\dots&\dots\\
0&0&0&\dots&\frac{x_{i+k-4}}{\sigma_{i+k-4}}&\frac{y_{i+k-3}}{\sigma_{i+k-3}}&0\\
-\frac{1}{\sigma_{i+k-2}}&0&0&\dots&0&\frac{x_{i+k-3}}{\sigma_{i+k-3}}&1\\
\frac{1}{\sigma_{i+k-2}}&\frac{1}{\lambda \sigma_{i+k-1}}&0&\dots&0&0&0\\
0&-\frac{1}{\lambda \sigma_{i+k-1}}&0&\dots&0&0&0\\
\end{array}\right),
$$
for $k\geq3$, where, as before, $\sigma_i=x_i+y_i$. Finally, $A_{j}(z)$ are given by
\begin{align}
\label{A12}
A_{1}(z) & = Z D_\sigma (\one_n - Z)^{-1}Z^{-r'},\\ 
\nonumber
A_{2}(z) & = Z^{r'+2} \left (D_x +  Z D_y\right ) D_\sigma^{-1} Z^{k-2},
\end{align}
where $D_\sigma = D_x + D_y = \diag(\sigma_1,\ldots, \sigma_n)$.
\end{proposition}

\begin{proof} The claim can be verified by a direct calculation using equations \eqref{mapxy}. It is worth pointing out, however, that expressions for $P_i(\lambda)$ and $A_{j}(z)$  were derived by re-casting elementary
network transformations that constitute $T_k$ as matrix transformation. We will provide an explanation for the $n\times n$ Lax representation and  leave the details for the $k\times k$ Lax representation as an instructive exercise for an inquisitive reader.

First, we rewrite equation \eqref{mapxy} for $T_k$  in terms of $D_x, D_y$:
\begin{align*}
D^*_x &= \left (Z^{-r'-1} D_x D_\sigma^{-1} Z^{r'+1} \right ) \left (Z^{r} D_\sigma Z^{-r} \right ) \\
&= Z^{-r'-1} \left ( D_x D_\sigma^{-1} \right ) \left (Z^{k-1} D_\sigma Z^{1-k} \right ) Z^{r'+1}, \\
D^*_y &= \left (Z^{-r'} D_y D_\sigma^{-1} Z^{r'} \right ) \left (Z^{r+1} D_\sigma Z^{-r-1} \right ) \\
&= Z^{-r'} \left ( D_x D_\sigma^{-1} \right ) \left (Z^{k-1} D_\sigma Z^{1-k} \right ) Z^{r'}, 
\end{align*}
which allows one to express  $Z^{r'+1} A_{k,n}^*(z) Z^{-r'-1}$ as
$$
 Z  \left(D_x D_\sigma^{-1} \left(Z^{k-1} D_\sigma Z^{1-k} \right) Z^{-1}  +     Z  D_y D_\sigma^{-1}Z^{-1}  \left(Z^{k} D_\sigma Z^{-k} \right) \right)
 Z^{k-1} \left ( \one_n - Z\right)^{-1}.   
$$
Denote  $Z^{r'+1} A_{k,n}^*(z) Z^{-r'-1}$ by $A^\sharp_{k,n}(z)$. If we find $A_1^\sharp(z)$, $A_2^\sharp(z)$ such that 
$$
A_{k,n}(z) = A^\sharp_1(z) A^\sharp_2(z),\qquad A^\sharp_{k,n}(z) = A^\sharp_2(z) A^\sharp_1(z), 
$$
then $A_1(z)=A_1^\sharp(z) Z^{-r'-1}$, $A_2(z)=A_2^\sharp(z) Z^{r'+1}$ will provide the desired Lax representation.

Consider the transformation of the network $\can'_{k,n}$ induced by performing type 3 Postnikov's move at all quadrilateral faces followed by performing type 1 Postnikov's move at all white-white edges. 
The resulting network is shown in Fig.~\ref{conjmoves}.

\begin{figure}[hbtp]
\centering
\includegraphics[height=3.3in]{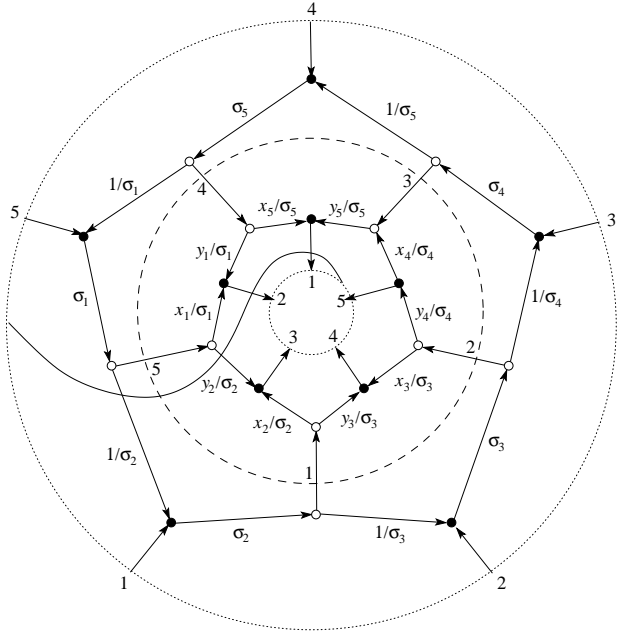}
\caption{The conjugate network $\can'_{3,5}$ after Postnikov's moves}
\label{conjmoves}
\end{figure}
 
To obtain  a  factorization of $A_{k,n}(z)$, we will view the latter network as a concatenation of two networks glued across a closed contour that intersects all white-white edges and no other edges
(in Fig.~\ref{conjmoves} it is represented by the dashed circle). Intersection points are labeled $1$ through $n$ counterclockwise with the label $i$ attached to
the point in the white-white edge incident to the edge of weight  $\sigma_{i+1}$. Furthermore, we adjusted the cut in such a way that it crosses the dashed circle through the segment between points labeled $1$ and $n$. 

Let $A^\sharp_1(z)$ and $A^\sharp_2(z)$ be boundary measurement matrices associated with the outer and the inner networks obtained this way. Clearly, $A_{k,n}(z) = A^\sharp_1(z) A^\sharp_2(z)$. To be able to perform the last sequence of transformation involved in $T_k$, namely to apply type 3 Postnikov's move to all black-black edges,
we have to cut the torus in a different way:
 we first need to separate two networks along the dashed circle and then glue the outer boundary of the outer network to the inner boundary of the inner network matching the labels of boundary vertices. But this means that $A^\sharp_{k,n}(z) = A^\sharp_2(z) A^\sharp_1(z)$.

It remains to check that expressions for  $A^\sharp_1(z)$, $A^\sharp_2(z)$ are consistent with~\eqref{A12}. Note that the cut crosses the dashed circle between the intersection points labeled $n-r'-1$ and $n-r'$. The network that corresponds to $A^\sharp_1(z)$ contains a unique oriented cycle of weight $z$ and, for any $i, j$, a unique simple directed path from source to sink $j$ that
contains the edge of weight $\sigma_{i+1}$.  This path does not intersect the cut if $i \leq j$, otherwise it intersects the cut once. Thus, the $(i,j)$ entry of $ A^\sharp_1(z) $ is equal
$ \frac {\sigma_{i+1}} {1+z} $ for $i \leq j$ and $ -\frac {\sigma_{i+1} z} {1+z} $ for $i > j$, which means that $A^\sharp_1(z) = Z D_\sigma Z^{-1}( \one_n - Z)^{-1}$ as needed.

The network that corresponds to $A^\sharp_2(z)$ contains no oriented cycles. The source $i$ is connected by  a directed path of weight $x_i/\sigma_i$  to the  sinks  $i+k-1$ and  by  a directed path of weight $y_{i+1}/\sigma_{i+1}$  to the  sink $i+k$. The former path intersects the cut if and only if 
$n-k+2 \leq i\leq n$, and the latter path intersects the cut if and only if $  n-k+1 \leq i\leq n$. We conclude that $A^\sharp_2(z)= Z\left ( D_x+ Z D_y\right ) D_\sigma^{-1} Z^{k-2}$, as needed.
\end{proof}

 In view of Proposition \ref{zero-curv}, the preservation of spectral invariants of $M_{k,n} (\lambda)$
(called, in this context, the {\em monodromy matrix\/}) and $A_{k,n} (\lambda)$ is obvious. In particular, 
$$
M_{k,n}^*= P_1 L_r \cdots L_n L_1\cdots L_{r-1} P_1^{-1}.
$$

\begin{remark} 
1. Two representations we obtained for $T_k$ give an example of what in integrable
systems literature is called {\em dual Lax representations}. 
A general technique for constructing integrable systems possessing such representations
based on dual moment maps was developed in~\cite{AHH}. 

2.  For  $k  = 3$, we obtain a $3\times 3$ zero-curvature representation for the pentagram map
alternative to the one given in \cite{So}.
\end{remark}

\subsection{Complete integrability}
\begin{theorem}\label{main}
  The map $T_k$ is completely integrable.
\end{theorem}

\begin{proof} Proposition~\ref{invariants} shows that spectral invariants of $M_{k,n}(\lambda)$ (equivalently,  
by Corollary~\ref{twocharpol}, spectral invariants of $A_{k,n}(z)$) 
are conserved quantities for $T_k$, while Theorem~\ref{cyltotor} and Proposition~\ref{netbracket} imply
that conserved quantities Poisson-commute. To establish complete integrability, we need to prove that this Poisson-commutative family is maximal. 

By Proposition \ref{netbracket},
the number of Casimir functions for our Poisson structure is $2 d$, where $d=\gcd(k-1, n)$. Therefore, we need to show that among spectral invariants of $A_{k,n}(z)$ there are $n+d$ independent functions of ${\bf x}=(x_i)_{i=1}^n$, ${\bf y}=(y_i)_{i=1}^n$. Furthermore, among Casimir functions described in~\eqref{casimirs} there are $d$ independent functions that depend only on  
$\bf y$. Hence it suffices to prove that
gradients of functions $f_s({\bf x}, {\bf y})=\frac {1} {s+1}\Trace A^{s+1}_{k,n}(z)$,  $s=0,\ldots, n-1$, 
viewed as functions  of $\bf x$ are linearly independent for almost all ${\bf x}, {\bf y}$ or, equivalently, 
since $f_s({\bf x}, {\bf y})$ are polynomials, for at least one point ${\bf x}, {\bf y}$. 
Using the formula for variation of traces of powers of a square matrix
$$
\delta \Trace A^{s+1} = \Trace \sum_{i=0}^s A^i \delta A A^{s-i} = (s+1) \Trace (\delta A A^s), 
$$
we deduce that the $i$th component of the gradient
$\nabla_{\bf x} f_s({\bf x}, {\bf y})$ with respect to $\bf x$ is equal to the $i$th diagonal element of the matrix
$$
Z^{k-1} (\one_n - Z)^{-1} \left ((D_x + D_y Z) Z^{k-1} (\one_n - Z)^{-1}\right )^s. 
$$
In particular,
\begin{align*}
\left(\nabla_{\bf x} f_s({\bf x}, - {\bf x})\right)_i &= \left ( Z^{k-1} (\one_n - Z)^{-1} \left (D_x Z^{k-1}\right )^s\right )_{ii}\\
&= {\frac {1}{1+z} }\left ( Z^{k-1} (\one_n + \ldots  + Z^{n-1}) \left (D_x Z^{k-1}\right )^s\right )_{ii} \\
&=  {\frac {1}{1+z} }\left ( Z^{k-1+l} \left (D_x Z^{k-1}\right )^s\right )_{ii}  \\
&=  {\frac {(-z)^{\lceil \frac {(s+1)(k-1)} {n}\rceil} }{1+z} }\prod_{\beta=1}^s x_{i-\beta (k-1)}.
\end{align*}
Here  the second line follows from~\eqref{secid}, in the third line $l + (s+1)(k-1)=0\bmod n$,
since only one of the powers of $Z$ present in the second line contributes to the diagonal, and it is the one 
with the exponent divisible by $n$,
the fourth line is obtained by
repeated application of~\eqref{firstid}, and the index $i-\beta(k-1)$ is understood $\bmod\ n$.
Prefactors depending on $z$ play no role in analyzing linear independence of  
$\nabla_{\bf x} f_s({\bf x}, - {\bf x})$, so we ignore them and form 
an $n\times n$ matrix $F=\left ( \prod_{\beta=1}^s x_{i-\beta (k-1)} \right )_{s,i=1}^n$. We need to show that $\det F$ is generically nonzero. We  further specialize by setting $x_{d+1}=\cdots=x_n=1$, then columns  of $F$ become
$$
F_{i+\alpha (k-1)} =\mbox{col}( \underbrace{1,\ldots,1}_{\alpha}, \underbrace{x_i,\ldots,x_i}_{n/d}, \underbrace{x^2_i,\ldots, x^2_i}_{n/d},\ldots )
$$ 
for $i=1,\ldots, d$, $\alpha= 1,\ldots, n/d$. Now the standard argument (akin to the one used in computing 
Vandermonde determinants)
shows that up to a sign $\det F$  is equal to $\prod_{i=1}^d (x_i -1)^{n/d - 1} \prod_{1\leq i < j \leq d} (x_i - 
x_j)^{n/d}$. The proof is complete.
\end{proof}

\begin{corollary}
Any rational function of $I_{ij}$ that is homogeneous of degree zero
in variables $x,y$, depends only on $p, q$ and is preserved by the map $\overline T_k$. On the level set $\{ \Pi
p_i q_i =1 \}$, such functions generate
a complete involutive family of integrals for the map $\barT_k$.
\label{completepq}
\end{corollary}

\begin{remark} In general,
these functions define a continuous integrable system on level sets
of the form $\{ \Pi p_i= c_1, \Pi q_i =c_2 \}$, and the map $\barT_k^{(c)}$ intertwines
the flows of this system on different level sets lying on the same hypersurface $c_1c_2=c$.
Numerical evidence suggests that $\barT_k^{(c)}$ is not integrable whenever $c\ne 1$.
Indeed, the left part of Fig.~\ref{IntegrableVsNonIntegrable} demonstrates the typical integrable behavior while the right part clearly shows at least three accumulation points which contradicts integrability. 

\begin{figure}[hbtp]
\centering
\includegraphics[width=2.2in]{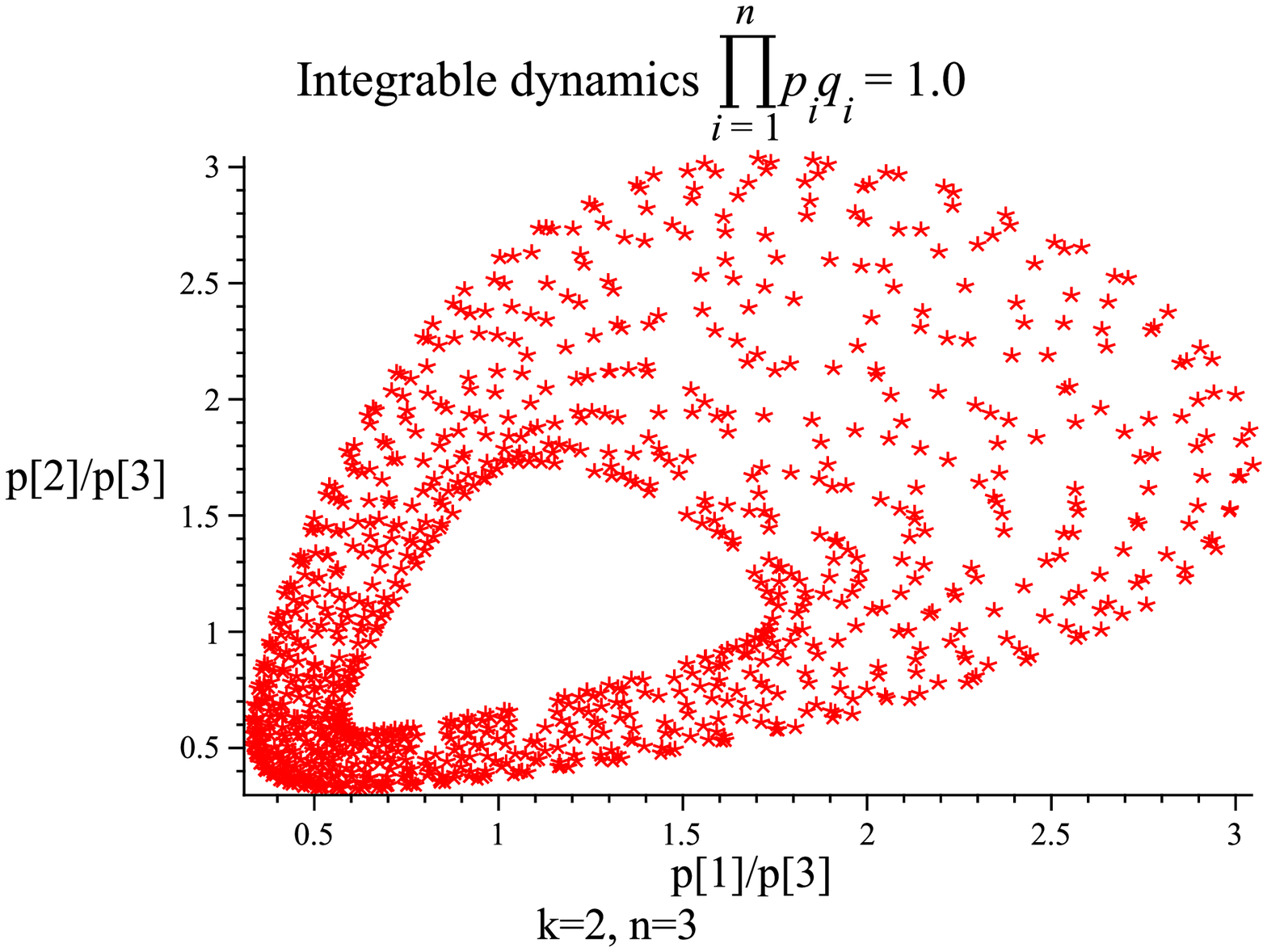}
\includegraphics[width=2.5in]{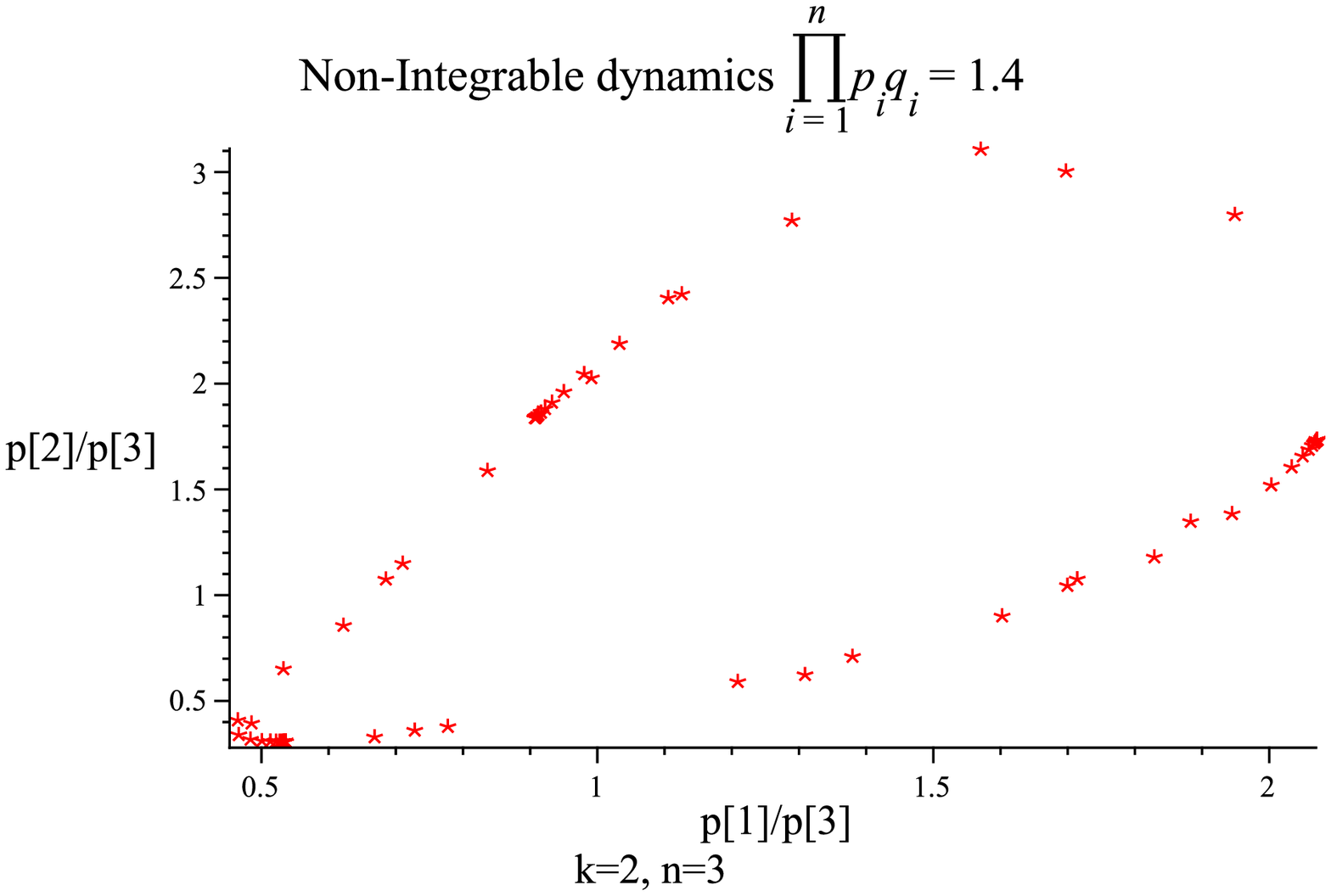}
\caption{Integrable and non-integrable dynamics}
\label{IntegrableVsNonIntegrable}
\end{figure}

\end{remark}

\subsection{Spectral curve}
We could have deduced independence of the integrals of $T_k$ from the properties of the {\em spectral curve\/}
\begin{equation}
\label{speccurve}
\frap_{\bar\can_{k,n}}(\lambda,z)=
\det (I_k  + z M_{k,n})(\lambda)) = (1+z) \det (I_n  + \lambda A_{k,n}(z)) =\sum_{i,j} I_{ij}\lambda^i z^j = 0
\end{equation}
in the spirit of \cite{OPRS}, Section~9. 
We will briefly discuss this spectral curve here to point out parallels between our approach 
and that of \cite{GoK}. 
There, the starting point for  constructing an integrable system is a centrally symmetric polygon 
$\Delta$ with vertices in $\mathbb{Z}^2$ which gives rise
to a dimer configuration on a torus  whose {\em partition function\/}  
serves as a generating function for Casimirs and integrals in involution and
 defines an algebraic curve with a Newton polygon $\Delta$.

\begin{proposition}
\label{Newton}
The Newton polygon $\Delta$ of the spectral curve~\eqref{speccurve}
is a parallelogram with vertices $(0,0), (0,1), (n,k), (n, k-1)$.
\end{proposition}

\begin{proof} Obviously, $(0,0), (0, 1)\in \Delta$. Since 
$$
\det A_{k,n}(z) = \frac{(-z)^{k-1}}{1+z} \left(\prod_i x_i - (-1)^n z\prod_i y_i\right)
$$ 
by~\eqref{dual_bm}, $(n,k), (n, k-1)$ are in $\Delta$. 

Let $(i,j)\in\Delta$; by the construction of 
$\frap_{\bar\can_{k,n}}(\lambda,z)$, it corresponds to a family $\cac$ of disjoint simple cycles 
in $\bar\can_{k,n}$ that has in total $i$ intersection points with $\rho$ and $j$ intersection points with $\gamma$.
Consider a simple directed cycle $C_t\in\cac$ that intersects the cut $\rho$ 
at points numbered $\alpha_1,\ldots, \alpha_{i_t+1}=\alpha_1$ (we list them in the order they appear along $C_t$
 starting with the smallest number and denote this list $A(C_t)$). 
 It is convenient to visualize $C_t$ using the network $\can'_{k,n}$. 
 Then one can see that $\alpha_{s+1}$ can be expressed as $\alpha_{s+1} = \alpha_s  + \beta_s 
 \bmod n$, where $s = 1, \ldots, i_t$ and $\beta_s \in [k-1, n+k-2]$, see Fig.~\ref{conjugate}. We thus have 
 $i_t (k-1) \leq \beta_1 + \cdots + \beta_{i_t} = j_t n$ for some $j_t$, which means that $C_t$ 
 intersects the rim $\gamma$ exactly  $j_t$ times. Consequently, $i_t$ and $j_t$ are subject to the 
 inequality $ \frac {k-1} {n} i_t \leq j_t$. Summing up over all cycles in $\cac$ and taking into account
 that $\sum i_t=i$, $\sum j_t=j$,  we get $ \frac {k-1} {n} i \leq j$.
 
 On the other hand, each shift $\beta_s$ as above prohibits at least $\beta_s-k$ indices 
from entering  the set $A(C)$ for any cycle $C\in\cac$; more exactly, if $\beta_s>k$ then the prohibited
indices are $\alpha_s+k,\dots \alpha_s+\beta_s-1$, otherwise there are no such indices. So,
 totally at least $j_tn-i_tk$ indices are prohibited.
 Clearly, the sets of prohibited indices for distinct cycles are disjoint. Therefore, altogether 
at least $jn-ik$ indices are prohibited and $i$ indices are used, hence $jn-i(k-1)\leq n$. 
Thus, if $(i,j) \in \Delta$, then $ \frac {k-1} {n} i \leq j \leq \frac {k-1} {n} i + 1$, which, together with an obvious inequality $0\leq i \leq n$, proves the claim.
\end{proof}

There are $2 (d +1)$ integer points on the boundary of $\Delta$: 
$$
 \left(l \frac {n} {d}, l \frac {k-1} {d}\right),\qquad \left(l \frac {n} {d}, l \frac {k-1} {d} +1\right),
 \quad l= 0,\ldots, d.
 $$ 
 The number of interior integer points (equal to the genus of the spectral curve) is $n-d$. The coefficient of 
 $\frap_{\bar\can_{k,n}}(\lambda,z)$ that corresponds to a point of the first type is a sum where each term is  
 the product of weights of $l$ disjoint cycles, each of them characterized uniquely by $i=n/d$, $\alpha_1\in [1,d]$, and $\beta_1 =\ldots = \beta_{n/d} = k-1$ (see Fig.~\ref{conjugate}). The weight of such cycle is
the Casimir function $\prod_{s=0}^{n/d-1} x_{\alpha_1 + s (k-1)}$, cp.~with the first expression in~\eqref{casimirs}.

On the other hand,  any term
contributing to the coefficient corresponding to a point of the second type is the weight of a collection of 
disjoint cycles that is represented in  $\can'_{k,n}$
by a unique collection of non-intersecting 
paths joining $ln/d$ sources $\alpha_1, \alpha_2, \ldots, \alpha_l, \alpha_1 +(k-1) , \alpha_2 + (k-1), \ldots, \alpha_l 
+(k-1), \ldots$ to sinks $\alpha_2 + (k-1) , \alpha_3 +(k-1) ,\ldots,  \alpha_l + (k-1), \alpha_1 +2 (k-1), 
\ldots$, respectively, where $1\leq \alpha_1 < \ldots < \alpha_l \leq d$. The weight in question is the product of Casimir functions $\prod_{s=0}^{n/d-1} y_{\alpha_t + s (k-1)}$ for $t=1,\ldots, l$, cp.~with the second expression in~\eqref{casimirs}.

Thus, just like in \cite{GoK}, interior points of $\Delta$ correspond to independent integrals while
integer points on the boundary of $\Delta$ correspond to Casimir functions.

\section{Geometric interpretation} \label{geom}

In this section we give a geometric interpretations of the maps $T_k$. The cases $k\geq 3$ and $k=2$ are different and are
treated separately.
\smallskip

\subsection{The case $k\geq 3$}

\subsubsection{Corrugated  polygons and generalized higher pentagram maps}
As we already mentioned, a {\it twisted $n$-gon\/} in a projective space is
 a sequence of points $V=(V_i)$ such that $V_{i+n}=M(V_i)$ for all $i \in \Z$ and some fixed projective transformation $M$  called the {\it monodromy}. The projective group naturally acts on the space of twisted $n$-gons. Let ${\mathcal P}_{k,n}$ be the
 space of projective equivalence classes of generic twisted $n$-gons in $\RP^{k-1}$, where ``generic" means that every $k$ consecutive vertices do not lie in a projective subspace. 
Clearly, the space ${\mathcal P}_{k,n}$ has dimension $n(k-1)$.

We say that  a twisted polygon $V$ is {\it corrugated\/} if, for every $i$, the vertices $V_i$, $V_{i+1}$, 
$V_{i+k-1}$ and $V_{i+k}$ span a projective plane. The projective group preserves the space of corrugated polygons. 
Projective equivalence classes of corrugated polygons constitute an algebraic subvariety of the moduli space of polygons in the projective space. Note that a polygon in $\RP^{2}$ is automatically corrugated.

Denote by ${\mathcal P}^0_{k,n}\subset \caP_{k,n}$ the space of projective equivalence classes of  corrugated polygons satisfying the additional genericity assumption that, for every $i$,
every three out of the four vertices $V_i$, $V_{i+1}$, $V_{i+k-1}$ and $V_{i+k}$ are not collinear.

\begin{lemma}
One has: ${\rm dim}\ {\mathcal P}^0_{k,n} = 2n$.
\end{lemma}

\begin{proof}
As it was already mentioned, ${\rm dim}\ {\mathcal P}_{k,n} = n(k-1)$. For each $i=1,\ldots,n$, one has a constraint: the vertex $V_{i+k}$ lies in the projective plane spanned by $V_i$, $V_{i+1}$, $V_{i+k-1}$. The codimension of a plane in $\RP^{k-1}$ is $k-3$, which yields $k-3$ equations. Thus 
$$
{\rm dim}\ {\mathcal P}^0_{k,n} = n(k-1) - n(k-3) = 2n,
$$
as claimed. 
\end{proof}

 The consecutive {\it $(k-1)$-diagonals\/} (the diagonals connecting $V_{i}$ and $V_{i+k-1}$) of a corrugated polygon intersect, and the intersection points form the vertices of a new twisted polygon: the $i$th vertex of this new polygon is the intersection of diagonals $(V_i,V_{i+k-1})$ and $(V_{i+1},V_{i+k})$.
 This $(k-1)$-diagonal map commutes with projective transformations, and hence one obtains a rational map
 $F_k: {\mathcal P}^0_{k,n}\to {\mathcal P}_{k,n}$. (Note that this rational map is well defined only on an open subset of ${\mathcal P}^0_{k,n}$ because the image polygon may be degenerate.)
$F_3$ is the pentagram map; the maps $F_k$ for $k>3$ are called {\it generalized higher pentagram maps}.

 \begin{remark} Corrugated polygons for $k>2$ were independently defined by M.Glick (private communication). 
\end{remark}

Given a corrugated polygon $V$, one can also construct a new polygon whose $i$th vertex is the 
intersection of the lines $(V_i,V_{i+1})$ and $(V_{i+k-1},V_{i+k})$. 
This defines a map $G_k: {\mathcal P}^0_{k,n}\to {\mathcal P}_{k,n}$. 

Similarly to above, one can define spaces of twisted and corrugated polygons in the dual projective space $(\RP^{k-1})^*$, as well 
as  dual analogs of the maps $F_k$ and $G_k$; in what follows, the objects in the dual space will be marked by an asterisk.  
Besides, we will need the notion of the projectively dual polygon. Let $V$ be a generic polygon in $\RP^{k-1}$. Each 
consecutive $(k-1)$-tuple of vertices spans a projective hyperplane, that is, a point of $(\RP^{k-1})^*$. This ordered collection of points represents the vertices of the dual polygon $W=V^*$; more exactly, the projective 
hyperplane spanned by $V_i, \dots, V_{i+k-2}$ represents the vertex $W_i$. We denote the projective duality map that takes $V$ to $W$ by $\Delta_k$.

\begin{proposition} \label{corrprop}
{\rm (i)} The image of a corrugated polygon under $F_k$ and under $G_k$ is a corrugated polygon.

{\rm (ii)} Up to a shift of indices by $k$, the maps $F_k$ and $G_k$ are inverse to each other.

{\rm (iii)} The polygon projectively dual to a corrugated polygon is corrugated.

{\rm (iv)} Projective duality $\Delta_k$ conjugates the maps $F_k$ and $G_k^*$.
\end{proposition}

\begin{figure}[hbtp]
\centering
\includegraphics[height=1.8in]{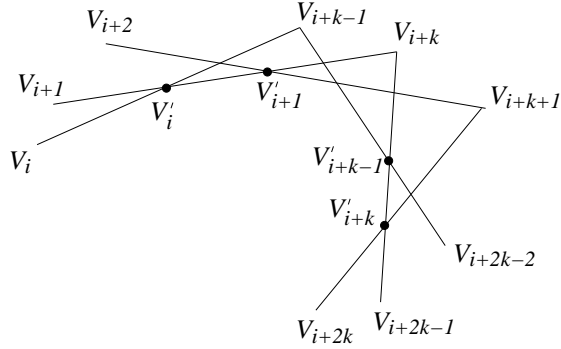}
\caption{The maps $F_k$ and $G_k$}
\label{lines}
\end{figure}

\begin{proof}
{\rm (i)} 
Let $V_i'$ denote the $i$th vertex of $F_k(V)$. We claim  that the vertices $V'_i$, $V'_{i+1}$, 
$V'_{i+k-1}$, $V'_{i+k}$ lie in a projective plane.
Indeed, $V'_i$ and $V'_{i+1}$ belong to the line $(V_{i+1},V_{i+k})$, and $V'_{i+k-1}$ and $V'_{i+k}$ to the 
line $(V_{i+2k-1},V_{i+k})$. These lines interest at point $V_{i+k}$, hence the the points $V'_i$, $V'_{i+1}$, 
$V'_{i+k-1}$, $V'_{i+k}$ are coplanar, see Fig.~\ref{lines}.

The argument for the map $G_k$ is analogous.

{\rm (ii)} Follows immediately from the definition of $F_k$ and $G_k$, see Fig.~\ref{lines}.

{\rm (iii)} Lift points in $\RP^{k-1}$ to vectors in $\R^k$; the lift is not unique and is defined up to a multiplicative factor. 
We use tilde to indicate a lift of a point. A lift of a twisted polygon is also twisted: $\widetilde V_{i+n} = \widetilde M (\widetilde V_i)$ for all $i$, where $\widetilde M \in GL(k,\R)$ is a lift of the monodromy.

Fix a volume form in $\R^k$. Then $(k-1)$-vectors are identified with covectors. In particular, if $V_i,\ldots,V_{i+k-2}$ are points in $\RP^{k-1}$ spanning a hyperplane then the respective point $W_i$ of the dual space $(\RP^{k-1})^*$ lifts to $\widetilde W_i=\widetilde V_i \wedge\dots \wedge \widetilde V_{i+k-2}$.

Let $V$ be a generic corrugated polygon. We need to prove that the $(k-1)$-vectors $\widetilde W_i$, $\widetilde W_{i+1}$, 
$\widetilde W_{i+k-1}$ and $\widetilde W_{i+k}$ are linearly dependent for all $i$. 

Since the polygon $V$ is corrugated, $\widetilde V_{i+2k-2}\in \operatorname{span}(\widetilde V_{i+2k-3},\widetilde V_{i+k-2},\widetilde V_{i+k-1})$, and hence
$$
\widetilde W_{i+k}=\widetilde V_{i+k}\wedge\dots\wedge \widetilde V_{i+2k-2}\in \operatorname{span} (\widetilde V_{i+k-2}\wedge \widetilde V_{i+k}\wedge\dots\wedge \widetilde V_{i+2k-3},\widetilde W_{i+k-1}).
$$
In its turn, 
\begin{equation*}
\begin{split}
\widetilde V_{i+k-2}\wedge \widetilde V_{i+k}\wedge\dots\wedge \widetilde V_{i+2k-3}=
\widetilde V_{i+k-3}\wedge \widetilde V_{i+k-2}\wedge \widetilde V_{i+k}\wedge\dots\wedge \widetilde V_{i+2k-4}=\\
\dots=\widetilde V_{i+1}\wedge\dots\wedge \widetilde V_{i+k-2}\wedge \widetilde V_{i+k} \in
\operatorname{span}(\widetilde V_i \wedge\dots \wedge \widetilde V_{i+k-2},\widetilde V_{i+1} \wedge\dots \wedge \widetilde V_{i+k-1}).
\end{split}
\end{equation*}
The latter two $(k-1)$-vectors are $\widetilde W_i$ and $\widetilde W_{i+1}$, as needed.

{\rm (iv)} We need to prove that $G_k^*\circ \Delta_k = \Delta_k\circ F_k$.

As before, we argue about lifted vectors. Let $\widetilde V$ be a lifted polygon, and let $W$ be the
polygon dual to $V$. 
The vertices of a lift of the polygon $W''=G^*_k(W)$ are represented by vectors in $\operatorname{span}(\widetilde W_i, \widetilde 
W_{i+1}) \cap \operatorname{span} (\widetilde W_{i+k-1}, \widetilde W_{i+k})$. Let $\widetilde W''_i$ be a vector spanning this line. Then 
$$
\widetilde W''_i = \alpha \widetilde W_i + \beta \widetilde W_{i+1} = \lambda \widetilde W_{i+k-1} + \mu\widetilde  W_{i+k}
$$
for some coefficients $\alpha, \beta, \lambda, \mu$. Using the definition of the points $W_i$, we have:
\begin{equation} \label{twoways}
\begin{split}
\widetilde W''_i = (\alpha \widetilde V_i \pm \beta \widetilde V_{i+k-1}) \wedge \widetilde V_{i+1}\wedge \ldots \wedge \widetilde V_{i+k-2} \\
=(\lambda \widetilde V_{i+k-1} \pm \mu \widetilde V_{i+2k-2}) \wedge \widetilde V_{i+k}\wedge \ldots \wedge \widetilde V_{i+2k-3}.
\end{split}
\end{equation}

On the other hand, the vertex $V'_i$ of $F_k(V)$ is the intersection point of the lines $(V_i,V_{i+k-1})$ and $(V_{i+1}, V_{i+k})$. Let $\widetilde V'_i$ be a lift of this point; then
$\widetilde V'_i = s \widetilde V_{i+1} + t \widetilde V_{i+k}$ where $s$ and $t$ are coefficients. We want to show that $\Delta_k(V'_i)=W''_i$, that is, 
using the identification of $(k-1)$-vectors with covectors, that $\widetilde V'_i \wedge \widetilde W''_i=0$. Indeed, in view of~\eqref{twoways},
\begin{equation*}
\begin{split}
\widetilde V'_i \wedge \widetilde W''_i
 = s \widetilde V_{i+1} \wedge (\alpha \widetilde V_i \pm \beta \widetilde V_{i+k-1}) \wedge \widetilde V_{i+1}\wedge \ldots \wedge \widetilde V_{i+k-2} +\\
 t \widetilde V_{i+k}\wedge (\lambda \widetilde V_{i+k-1} \pm \mu \widetilde V_{i+2k-2}) \wedge \widetilde V_{i+k}\wedge \ldots \wedge \widetilde V_{i+2k-3} =0,
\end{split}
\end{equation*}
as claimed.
\end{proof}

\begin{remark}
Let us briefly mention a natural continuous analog of corrugated polygons and the generalized higher pentagram maps. A twisted curve in $\RP^{k-1}$ is a map $\gamma: \R \to \RP^{k-1}$ such that $\gamma(x+1)=M(\gamma(x))$ for all $x$ where $M$ is a fixed projective transformation. The projective group naturally acts on twisted curves, and one considers the projective equivalence classes.

A  curve is called $c$-corrugated if the tangent lines at points $\gamma(x)$ and $\gamma(x+c)$ are coplanar (not skew) 
for all $x\in\R$.  We claim that the line connecting points $\gamma(x)$ and $\gamma(x+c)$ envelops a new curve, $\bar \gamma(x)$. Indeed, the coplanarity condition implies that 
$\gamma(x+c) - \gamma(x) = u(x) \gamma'(x) + v(x) \gamma'(x+c)$ for some functions $u$ and $v$. Then the envelope $\bar \gamma$ is given by the equation
$$
\bar \gamma(x) = \frac{u(x)}{u(x)+v(x)} \gamma(x) + \frac{v(x)}{u(x)+v(x)} \gamma(x+c),
$$
as can be easily verified by differentiation. 

Thus we obtain a map $F_c:\gamma \mapsto \bar \gamma$. It is even easier to describe a map $G_c$ defined by 
a $c$-corrugated curve $\gamma$: it is traced by the intersection points of the tangent lines at points $\gamma(x)$ and $\gamma(x+c)$. 
These maps commute with projective transformations. Of course, the notion of corrugated curve is interesting only when $k-1\ge 3$. 

An analog of Proposition~\ref{corrprop} holds; in particular, $F_c (\gamma)$ is again a $c$-corrugated curve. The dynamics of the projective equivalence classes of corrugated curves under the transformations $F_c$ and $G_c$ is an interesting subject; we do not dwell on it here.
\end{remark}

\subsubsection{Coordinates in the space of corrugated polygons} 
Now we introduce coordinates in  ${\mathcal P}^0_{k,n}$. 

\begin{proposition} \label{coord}
One can lift the vertices of a generic corrugated polygon $V$ so that, for all $i$, one has:
\begin{equation} \label{recurr}
\widetilde V_{i+k}=y_{i-1} \widetilde V_{i} + x_i \widetilde V_{i+1} + \widetilde V_{i+k-1},
\end{equation}
where $x_i$ and $y_i$ are $n$-periodic sequences. Conversely, $n$-periodic sequences $x_i$ and $y_i$ uniquely determine the projective equivalence class of a twisted corrugated $n$-gon in $\RP^{k-1}$.
\end{proposition}

\begin{proof} 
Consider a lifted  twisted polygon $\widehat V$. Since $V$ is corrugated, one has
\begin{equation} \label{recurrgen}
\widehat V_{i+k}=a_{i+k-1} \widehat V_{i+k-1} + b_{i+1} \widehat V_{i+1} + c_i \widehat V_i
\end{equation}
for all $i$. The sequences $a_i, b_i$ and $c_i$ are $n$-periodic and, due to the genericity assumption, none of these coefficients vanish.

We wish to choose the lift in such a way that the coefficient $a$  identically equals~1. Rescale:
$\widehat V_i = \lambda_i \widetilde V_i,$
where $\lambda_i \neq 0$. Then for (\ref{recurrgen}) to become (\ref{recurr}), the following recurrence should hold for the scaling factors:
$\lambda_{i+1}=a_i \lambda_i$.
Set $\lambda_0=1$ and determine $\lambda_i$,  $i\in\Z$, by the recurrence. 

After this rescaling, the coefficients change as follows:
$$
c_i \mapsto \frac{c_i}{a_i a_{i+1}\cdots a_{i+k-1}},\qquad b_{i+1} \mapsto \frac{b_{i+1}}{a_{i+1}\cdots a_{i+k-1}}.
$$ 
Hence
\begin{equation} \label{newxy}
x_i=\frac{b_{i+1}}{a_{i+1}\cdots a_{i+k-1}},\qquad y_i=\frac{c_{i+1}}{a_{i+1}\cdots a_{i+k}}.
\end{equation}
Thus we obtain recurrence~\eqref{recurr} with $n$-periodic coefficients uniquely determined by the projective equivalence class of the twisted corrugated polygon.

Conversely, given $n$-periodic sequences $x_i$ and $y_i$, choose a frame $\widetilde V_0,\ldots, \widetilde V_{k-1}$ in $\R^k$ and use recurrence~\eqref{recurr} to construct a bi-infinite sequence of vectors $\widetilde V_i$. The periodicity of the sequences $x_i$ and $y_i$ implies that the polygon $\widetilde V$ is twisted, and relation~\eqref{recurr} implies that it is corrugated. A different choice of a frame results in a linearly equivalent polygon and, after projection to $\RP^{k-1}$, in a projectively equivalent polygon $V$. 
\end{proof}
 
The next theorem interprets the map $T_k$ as the generalized higher pentagram  map $F_k$.

\begin{theorem} \label{corrug}

{\rm (i)} In the $({\bf x},{\bf y})$-coordinates, the maps $F_k$ and $G_k$ coincide with $T_k$ and $T_k^{-1}$ up to a shift of indices. 
More exactly, if $S_t:{\mathcal P}^0_{k,n}\to {\mathcal P}^0_{k,n}$ is the shift by $t$ in the positive direction, then
$F_k=T_k\circ S_{r'+1}$ and $G_k=T_k^{-1}\circ S_{r+1}$.

{\rm (ii)} In the $({\bf x},{\bf y})$-coordinates, the projective duality $\Delta_k$ coincides with $D_k$ up to a sign and a shift of indices.
More exactly, $\Delta_k=(-1)^k D_k\circ S_{r'}$. 
\end{theorem}

\begin{proof}
 (i) Let $\widetilde V$ be a  polygon in $\R^k$ satisfying~\eqref{recurr}. Then
 \begin{equation*}
 \widetilde V_{j+k-1}+ y_{j-1} \widetilde V_j   = \widetilde V_{j+k} -x_j \widetilde V_{j+1} 
\end{equation*}
for $j=i, i+1, i+k-1, i+k$. 
It follows that, as lifts of points $V'_i, V'_{i+1}, V'_{i+k-1}, V'_{i+k}$ in Fig.~\ref{lines}, one may set:
 \begin{equation} \label{foureq}
 \begin{aligned}
\widetilde V'_i=\widetilde V_{i+k} -x_i \widetilde V_{i+1}, &\qquad 
\widetilde V'_{i+1} =\widetilde V_{i+k}+ y_{i} \widetilde V_{i+1},\\
\widetilde  V'_{i+k-1} = \widetilde V_{i+2k-1} -x_{i+k-1} \widetilde V_{i+k}, &\qquad
\widetilde V'_{i+k} = \widetilde V_{i+2k-1}+ y_{i+k-1} \widetilde V_{i+k}.
\end{aligned}
\end{equation}

One has a linear relation
$$
\widetilde V'_{i+k} = a \widetilde  V'_{i+k-1} + b \widetilde V'_{i+1} + c\widetilde V'_i. 
$$
Substitute from~\eqref{foureq} to obtain
$$
\widetilde V_{i+2k-1}+ y_{i+k-1} \widetilde V_{i+k} = a (\widetilde V_{i+2k-1} -x_{i+k-1} \widetilde V_{i+k}) + b (\widetilde V_{i+k}+ y_{i} \widetilde V_{i+1}) + c (\widetilde V_{i+k} -x_i \widetilde V_{i+1})
$$
and use linear independence of the vectors $\widetilde V_{i+1}, \widetilde V_{i+k}, \widetilde V_{i+2k-1}$ to conclude that
$$
a=1,\qquad b=x_i \frac{x_{i+k-1}+y_{i+k-1}}{x_i+y_i},\qquad c=y_i \frac{x_{i+k-1}+y_{i+k-1}}{x_i+y_i}.
$$
Thus the vectors $\widetilde V'_i, \widetilde V'_{i+1}, \widetilde  V'_{i+k-1}, \widetilde V'_{i+k}$ satisfy 
recurrence~\eqref{recurr} with the coefficients
$$
x'_i=x_i \frac{x_{i+k-1}+y_{i+k-1}}{x_i+y_i},\qquad y'_i=y_{i+1} \frac{x_{i+k}+y_{i+k}}{x_{i+1}+y_{i+1}}.
$$
This differs from~\eqref{mapxy} only by shifting indices by $r'+1$.
 
The statement about $G_k$ follows immediately from  $F_k=T_k\circ S_{r'+1}$ and Proposition~\ref{corrprop}(ii).
 
{\rm (ii)}  We use the same notation as in Proposition \ref{corrprop}. Let $\widetilde V$ be a polygon in $\R^k$ 
satisfying~\eqref{recurr}, and $\widetilde W$ be a lift of the dual polygon.

Going by the proof of Proposition~\ref{corrprop}(iii), line by line, we obtain the equalities:
$$
\widetilde W_{i+k}=(-1)^k y_{i+k-3}  \widetilde V_{i+k-2}\wedge \widetilde V_{i+k}\wedge\dots\wedge \widetilde V_{i+2k-3} + (-1)^k x_{i+k-2} \widetilde W_{i+k-1},  
$$   
and
$$
\widetilde V_{i+k-2}\wedge \widetilde V_{i+k}\wedge\dots\wedge \widetilde V_{i+2k-3} = (-1)^k y_{i-1} \cdots y_{i+k-4} 
\widetilde W_i + y_i \cdots y_{i+k-4} \widetilde W_{i+1}.
$$   
Hence
$$
\widetilde W_{i+k}=(-1)^k x_{i+k-2} \widetilde W_{i+k-1} + (-1)^k y_i \cdots y_{i+k-3} \widetilde W_{i+1} + y_{i+1} \cdots y_{i+k-3} \widetilde W_i.
$$    
Using  formulas~\eqref{newxy}, we find
$$
x_i^*= (-1)^k \frac{y_i \cdots y_{i+k-3}}{x_i \cdots x_{i+k-2}},\qquad y_i^*= (-1)^k \frac{y_i \cdots y_{i+k-2}}{x_i \cdots x_{i+k-1}},
$$   
which differs from~\eqref{mapdk} only by the sign and the shift of indices by $r'$.
\end{proof}

\begin{remark}
1. In view of Theorem~\ref{corrug} and Proposition~\ref{tinv}(i,ii), we may identify the maps $T_k^\circ$ and $G_k^*$.

2. It would be interesting  to provide a geometric interpretation for Proposition~\ref{tinv}(iii), which connects the maps
$G_k^*$ and $F_{n-k+2}$.
\end{remark}

Statement (i) of Theorem~\ref{corrug}, along with Theorem~\ref{main}, implies that  the generalized higher pentagram map $F_k$ is completely integrable. 

Integrals of the pentagram map~\eqref{pentaformula} were constructed by R. Schwartz in~\cite{Sch3}. He observed that the map 
commutes with the scaling
$$
X_i \mapsto \lambda X_i, \quad Y_i \mapsto \lambda^{-1} Y_i,
$$
and that the conjugacy class of the monodromy of a twisted polygon is invariant under the map. Decomposing  the characteristic polynomial of the monodromy into the homogeneous components with respect to the scaling, yields the integrals.  

The next proposition shows that the integrals of the generalized higher pentagram map can be constructed in a similar way.

\begin{proposition}\label{intviamono}
The integrals of $F_k$ can be obtained from the monodromy $M$ as the homogeneous components of its characteristic polynomial. 
\end{proposition}

\begin{proof} By Theorem~\ref{corrug} and~\eqref{integrals}, it suffices to check that 
the matrix $M_{k,n}(\lambda)$ given by~\eqref{Mmatrix} is conjugate-transpose to the scaled monodromy $M(\lambda)$ of the respective twisted corrugated polygon.

Let us start with the monodromy.  We use vectors $\widetilde V_1,\dots,\widetilde V_k$ as a basis in the vector space $\R^k$, 
and express the next vectors using~\eqref{recurr}. When we come to $\widetilde V_{n+1},\dots,\widetilde V_{n+k}$, we get the monodromy matrix as a function of the coordinates ${\bf x}, {\bf y}$. This process is represented by the product of matrices $Q_i(\lambda)$ that encode one step $(\widetilde V_i\dots,\widetilde V_{i+k-1}) \mapsto (\widetilde V_{i+1}\dots,\widetilde V_{i+k})$. By~\eqref{recurr}
\begin{equation*}
Q_i(\lambda)=\begin{pmatrix}
0 & 1 & 0 & \dots & 0 & 0 \cr
0 & 0 & 1 & \dots & 0 & 0\cr
\dots & \dots & \dots & \dots & \dots & \dots \cr
0 & 0 & 0 & \dots & 0 & 1\cr
\lambda y_{i-1} & \lambda x_i & 0 &  \dots & 0 & 1
\end{pmatrix},
\end{equation*}
where $\lambda$ is the scaling factor,
and $M(\lambda)=Q_n(\lambda)Q_{n-1}(\lambda)\cdots Q_1(\lambda)$. 

Now consider $M_{k,n}(\lambda)^T$: by~\eqref{Mmatrix}, it is the product of $L_i(\lambda)^T$, $i=1,\dots,n$, in the reverse order,
with $L_i(\lambda)$ given by~\eqref{Laxmat}. Let $A_i(\lambda)$ be an auxiliary matrix
\begin{equation*}
A_i(\lambda)=\begin{pmatrix}
\lambda^{-1} & 0 & 0&\dots & 0 & 0 \cr
0 & 1 & 0 &\dots & 0 & 0\cr
0 & 0 & 1 &\dots & 0 & 0\cr
\dots & \dots & \dots & \dots & \dots & \dots \cr
x_i & 0 & 0 &  \dots & 0 & 1
\end{pmatrix},
\end{equation*}
then $L_i(-\lambda)^T = A_{i+1}(\lambda)^{-1} Q_{i+1}(\lambda) A_i(\lambda)$, and hence 
$$
M_{k,n}(-\lambda)^T=A_{n+1}(\lambda)^{-1}Q_{n+1}(\lambda)q_n(\lambda)\cdots Q_2(\lambda)A_1(\lambda). 
$$
Taking into account that sequences
$Q_i(\lambda)$ and $A_i(\lambda)$ are $n$-periodic by Proposition~\ref{coord}, we get
$$
M_{k,n}(-\lambda)^T= A_1(\lambda)^{-1}Q_1(\lambda)M(\lambda)Q_1(\lambda)^{-1}A_1(\lambda),
$$
 and the claim follows.
\end{proof}

We complete the discussion of coordinates in the space of corrugated polygons  by showing that 
the coordinates $(p_i, q_i)$ introduced in Section~\ref{pq} can be interpreted as cross-ratios of quadruples of collinear points. We use the following definition of  cross-ratio (out of six possibilities):
\begin{equation} \label{crs}
[a,b,c,d]=\frac{(a-b)(c-d)}{(a-d)(b-c)}.
\end{equation}

To a twisted corrugated $n$-gon $V$ we assign two $n$-periodic sequences of cross-ratios. Consider 
Fig.~\ref{lines} and Fig.~\ref{colines} (depicting the map $G_k$) and observe quadruples of collinear points. Their cross-ratios are related to the coordinates $(p_i,q_i)$ as follows.

\begin{figure}[hbtp]
\centering
\includegraphics[height=1.3in]{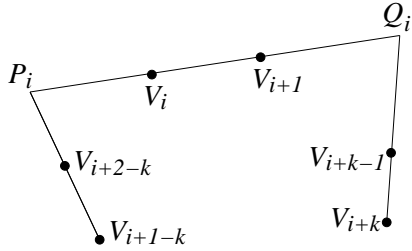}
\caption{Cross-ratios $q_i$}
\label{colines}
\end{figure}

\begin{proposition}
One has:
$$
p_i=[V_{i+1},V_i',V_{i+k},V_{i+1}'],\qquad q_{i-r-1}=[P_i,V_{i+1},Q_i,V_i]
$$
subject to $\prod_{i=1}^n p_iq_i=1$.
\end{proposition}

\begin{proof}
Consider the lifted polygons. We saw in the proof of Theorem~\ref{corrug} that
$$
\widetilde V'_i=\widetilde V_{i+k} -x_i \widetilde V_{i+1},\quad \widetilde V'_{i+1} =\widetilde V_{i+k}+ y_{i} \widetilde V_{i+1}.
$$
Choose the basis in the 2-plane so that 
$
\widetilde V_{i+1}=(1,0),\quad \widetilde V_{i+k}=(0,1).
$
Then $\widetilde V'_i=(-x_i,1)$, $\widetilde V'_{i+1} =(y_i,1)$, and 
$$
[V_{i+1},V'_i,V'_{i+1},V_{i+k}] =[\infty,-x_i,0,y_i]=\frac{y_i}{x_i} = p_i,
$$
the last equality due to Proposition~\ref{weights}(ii).

Likewise, one has
$$
P_i=\widetilde V_{i+1} - \widetilde V_i,\qquad Q_i=x_i\widetilde V_{i+1} + y_{i-1} \widetilde V_i
$$
in Fig.~\ref{colines}, and this yields
$$
[P_i,V_{i+1},Q_i,V_i]=\frac{x_i}{y_{i-1}}=q_{i-r-1},
$$
the last equality again due to Proposition~\ref{weights}(ii).

The condition on the product of all coordinates follows immediately.
\end{proof}

\subsubsection{ Higher pentagram maps on plane polygons} 
One  also has the {\it skip $(k-2)$-diagonal map\/} on twisted polygons in the projective plane.  Assume that the polygons are generic in the following sense: for every $i$, no three out of the four vertices $V_i,V_{i+1},V_{i+k-1}$ and $V_{i+k}$ are collinear. The skip $(k-2)$-diagonal map assigns to a twisted $n$-gon $V$ the twisted $n$-gon whose consecutive vertices are the intersection points of the lines $(V_i,V_{i+k-1})$ and $(V_{i+1},V_{i+k})$. We call these maps {\it higher pentagram maps\/} and denote them by 
$\bar F_k$. Assume that the ground field is $\C$.

Arguing as in the proof of Proposition \ref{coord}, we lift the points $V_i$ to vectors $\widetilde V_i\in \C^3$ so 
that~\eqref{recurr} holds. This provides a rational map $\psi$ from the moduli space of twisted polygons in $\CP^2$ to the $({\bf x,y})$-space, that is, to the moduli space of corrugated twisted polygons in $\CP^{k-1}$. On the  latter space, the generalized higher pentagram map $F_k$ acts.  The relation between these maps is as follows.

\begin{proposition} \label{diags}
{\rm (i)} The map $\psi$ is $k\choose 3$-to-one.

{\rm (ii)} The map $\psi$ conjugates $\bar F_k$ and $F_k$, that is, $\psi\circ \bar F_k=F_k\circ \psi$.

\end{proposition}
\begin{proof}
Given periodic sequences $x_i$, $y_i$, we wish to reconstruct a twisted $n$-gon in $\C^3$, up to a linear equivalence. To this end, let the first three vectors $\widetilde V_1, \widetilde V_2, \widetilde V_3$ form a basis, and choose $\widetilde V_4,\dots,\widetilde V_k$ arbitrarily, so far. Then $\widetilde V_{k+1}$, and all the next vectors, are determined by the recurrence~\eqref{recurr}. The monodromy $\widetilde M$ is determined by the condition that 
\begin{equation} \label{mono}
\widetilde M(\widetilde V_1)=\widetilde V_{n+1},\quad \widetilde M(\widetilde V_2)=\widetilde V_{n+2},\quad 
\widetilde M(\widetilde V_3)=\widetilde V_{n+3}.
\end{equation}
The twist condition is that 
\begin{equation} \label{twist}
\widetilde M(\widetilde V_j)=\widetilde V_{n+j},\quad j=4,\dots,k.
\end{equation}
If this holds, then $\widetilde M(\widetilde V_i)=\widetilde V_{n+i}$ for all $i$. Note that~\eqref{twist} gives $3(k-3)$ equations on that many variables (the unknown vectors being $\widetilde V_4,\dots,\widetilde V_k$). We shall see that these are quadratic equations and proceed to solving them.

The recurrence~\eqref{recurr} implies that $\widetilde V_q=\sum_{i=1}^k F_q^i \widetilde V_i$ where $F_q^i$ is a function of ${\bf x}, {\bf y}$. One has: $\widetilde V_j=v_j^1 \widetilde V_1 + v_j^2 \widetilde V_2 + v_j^3 \widetilde V_3$ where $v_j^1, v_j^2, v_j^3$ are the components of the vector $\widetilde V_j$ in the basis $\widetilde V_1, \widetilde V_2, \widetilde V_3$. 
Rewrite~\eqref{twist}, using~\eqref{mono}, as
$$
v_j^1 \left(\sum_{i=1}^k F_{n+1}^i \widetilde V_i\right) + v_j^2 \left(\sum_{i=1}^k F_{n+2}^i \widetilde V_i\right) + v_j^3 \left(\sum_{i=1}^k F_{n+3}^i \widetilde V_i\right) =
\sum_{i=1}^k F_{n+j}^i \widetilde V_i,
$$
or 
$$
A(\widetilde V_j)+\sum_{i=4}^k (F^i\cdot \widetilde V_j) \widetilde V_i = C_j +\sum_{i=4}^k g_j^i \widetilde V_i,\quad j=4,\ldots,k,
$$
where $A=(F^{\beta}_{n+\alpha})$, $\alpha,\beta=1,2,3$, is a $3\times3$ matrix, $F^i=(F_{n+1}^i, F_{n+2}^i, F_{n+3}^i)$, 
$C_j=(F_{n+j}^1,F_{n+j}^2,F_{n+j}^3)$, and $g_j^i=F_{n+j}^i$. Rewrite, once again, as 
\begin{equation} \label{eq}
A(\widetilde V_j)-C_j=\sum_{i=4}^k [g_j^i-(F^i\cdot \widetilde V_j)]\ \widetilde V_i.
\end{equation}
Let $B$ be the $k \times k$ matrix that has $A$ in the upper left corner, the vectors $-C_4, -C_5,\dots,-C_k$ to the right of $A$, the vectors $-(F_4)^t,\dots,-(F_k)^t$ below $A$, and the $(k-3)\times (k-3)$ matrix $G=g_j^i$ in the bottom right corner. The entries of $B$ are functions of ${\bf x}, {\bf y}$. Let $\xi_j,\ j=4,\dots,k$, be the $k$-dimensional vector $(V_j^1,V_j^2,V_j^3,0,\dots,0,1,0,\dots,0)$ with 1 at $j$-th position. Let $\mathcal V$ be the span of $\xi_j,\ j=4,\dots,k$.
\medskip

{\it Claim}: the system~\eqref{eq} is equivalent to the condition that $\mathcal V$ is a $B$-invariant subspace, that is, a fixed point of the action of $B$ on the Grassmannian $\operatorname{Gr} (k-3,k)$.

Indeed, one has: 
\begin{equation} \label{B}
B(\xi_j)=(A(\widetilde V_j)-C_j,g_j^4-(F^4\cdot \widetilde V_j),\dots,g_j^k-(F^k\cdot \widetilde V_j)).
\end{equation}
If (\ref{eq}) holds then 
\begin{equation*}
\begin{aligned}
B(\xi_j)&=\left(\sum_{i=4}^k [g_j^i-(F^i\cdot \widetilde V_j)]\ \widetilde V_i,g_j^4-(F^4\cdot \widetilde V_j),\dots,g_j^k-(F^k\cdot \widetilde V_j)\right)\\
&=\sum_{i=4}^k [g_j^i-(F^i\cdot \widetilde V_j)]\ \xi_i,
\end{aligned}
\end{equation*}
so $\mathcal V$ is $B$-invariant. 

Conversely, if $\mathcal V$ is $B$-invariant then $B(\xi_j)=\sum_{i=4}^k \alpha_i \xi_i$ with some coefficients $\alpha_i$. It follows from~\eqref{B} that $\alpha_i=g_j^i-(F^i\cdot \widetilde V_j)$, and that 
$$
A(\widetilde V_j)-C_j=\sum_{i=4}^k [g_j^i-(F^i\cdot \widetilde V_j)] \widetilde V_i,
$$ which is equation~\eqref{eq}. This proves the claim.

A generic linear transformation $B$ has a simple spectrum and $k$ one-dimensional eigenspaces. One has $k\choose 3$ invariant $(k-3)$-dimensional subspaces that can be parameterized as $\operatorname{span}(\xi_4,\dots,\xi_k)$. Thus one has $k\choose 3$ choices of vectors $\widetilde V_4,\dots,\widetilde V_k$ for given coordinates $x_i,y_i$. In other words, the mapping $\psi$ from the moduli space of twisted $n$-gons ${\mathcal P}_n$ to the ${\bf x}, {\bf y}$-space is $k\choose 3$-to-one. 
That this map conjugates the skip $(k-2)$-diagonal map $\bar F_k$ with the map $F_k$ is obvious, and we are done.
\end{proof}

It follows that if $I$ is an integral of the map $F_k$ then $I\circ \psi$ is an integral of the map $\bar F_k$. Thus the integrals (\ref{integrals}) provide integrals of the higher pentagram map.

\subsection{The case $k=2$: leapfrog map and circle pattern}

\subsubsection{Space of pairs of twisted $n$-gons in $\RP^1$}
Let ${\mathcal S}_n$ be the space whose points are pairs of twisted $n$-gons $(S^{-},S)$ in $\RP^1$ with the same monodromy. Here $S$ is a sequence of points $S_i \in \RP^1$, and likewise for $S^{-}$. One has: dim ${\mathcal S}_n=2n+3$. The group $PGL(2,\R)$ acts on ${\mathcal S}_n$.  Let $\varphi$ be the map from  ${\mathcal S}_n$ to the $({\bf x,y})$-space given by the formulas:
\begin{equation} \label{xyST}
x_i=\frac{(S_{i+1}-S^{-}_{i+2})(S^{-}_i-S^{-}_{i+1})}{(S^{-}_i-S_{i+1})(S^{-}_{i+1}-S^{-}_{i+2})}, \quad 
y_i=\frac{(S^{-}_{i+1}-S_{i+1})(S^{-}_{i+2}-S_{i+2})(S^{-}_i-S^{-}_{i+1})}{(S^{-}_{i+1}-S_{i+2})(S^{-}_i-S_{i+1})(S^{-}_{i+1}-S^{-}_{i+2})}.
\end{equation}

Recall that we use the cross-ratio defined by formula~\eqref{crs}.

\begin{proposition} \label{phi}
{\rm (i)} The composition of $\varphi$ with the projection $\pi$ is given by the formulas
$$
p_i=[S^{-}_{i+1},S_{i+1},S^{-}_{i+2},S_{i+2}],\quad  q_{i}=\frac{[S^{-}_i,S_{i+1},S_{i+2},S^{-}_{i+3}] [S^{-}_{i+1},S^{-}_{i+2},S_{i+2},S^{-}_{i+3}]}{[S^{-}_i,S^{-}_{i+1},S^{-}_{i+2},S^{-}_{i+3}][S^{-}_{i+1},S_{i+1},S_{i+2},S^{-}_{i+3}]}.
$$

 {\rm (ii)} The image of the map $\pi\circ\varphi$ belongs to the hypersurface $\prod_{i=1}^n p_iq_i=1$.

{\rm (iii)} The fibers of this maps are the $PGL(2,\R)$-orbits, and hence the $({\bf x,y})$-space is identified with the moduli space ${\mathcal S}_n/PGL(2,\R)$.
\end{proposition}

\begin{proof}
{\rm (i)} From Proposition~\ref{weights}, we have:
\begin{equation} \label{pqxy}
p_i=\frac{y_i}{x_i},\qquad  q_{i}=\frac{x_{i+1}}{y_{i}}.
\end{equation}
Then a direct computation using formulas~\eqref{xyST} for $x$ and $y$  yields the result.

{\rm (ii)} The sequences $x_i$ and $y_i$ are $n$-periodic. Multiplying  $p_i$ and $q_i$ from~\eqref{pqxy}, $i=1,\ldots, n$, the numerators and denominators cancel out, and the result follows.

{\rm (iii)} Put the sequences of points $S^-, S$ in the interlacing order:
$$
\ldots, S_i^-, S_i, S_{i+1}^-, S_{i+1}, S_{i+2}^-,\ldots
$$   
and consider the cross-ratios of the consecutive quadruples:
$$
p_{i-1}=[S_i^-,S_i,S_{i+1}^-,S_{i+1}],\qquad r_i=[S_i,S_{i+1}^-,S_{i+1}, S_{i+2}^-];
$$
the first equality was proved in (i), and the second is the definition of $r$.
The sequences $p_i$ and $r_i$ are $n$-periodic and they determine the projective equivalence class of the pair $(S^-,S)$. Thus we have a coordinate system $({\bf p,r})$ on  ${\mathcal S}_n/PGL(2,\R)$.

We wish to show that  $({\bf x,y})$ is another coordinate system. Indeed, one can express  $({\bf x,y})$ in terms of $({\bf p,r})$ and vice versa:
$$
x_i=\frac{r_i(1+p_{i-1})}{1+r_i},\quad y_i=\frac{r_i p_i (1+p_{i-1})}{1+r_i};\qquad p_i=\frac{y_i}{x_i},\quad r_i=\frac{x_{i-1}x_i}{x_{i-1}(1-x_i)+y_{i-1}}
$$
(we omit this straightforward computation).
\end{proof}

\subsubsection{Leapfrog transformation} 
Define  a transformation $\Phi$ of the space ${\mathcal S}_n$, acting as $\Phi(S^{-},S)=(S,S^{+})$, where $S^{+}$ is given by the following local ``leapfrog" rule: given a quadruple of points $S_{i-1}, S^{-}_i, S_i, S_{i+1}$, the point $S_i^{+}$ is the result of applying to $S^{-}_i$ the unique projective involution that fixes $S_i$ and interchanges $S_{i-1}$ and $S_{i+1}$. Clearly, $\Phi$ commutes with projective transformations.

The transformation $\Phi$ can be defined this way over any ground field, however in $\RP^1$ we can
interpret  the point $S_i^{+}$ as  the reflection of $S^{-}_i$ in $S_i$ in the projective metric on the segment $[S_{i-1},S_{i+1}]$, whence the name; see Fig.~\ref{leap}. Recall that the projective metric on a segment is the Riemannian metric whose isometries are the projective transformations preserving the segment, see~\cite[Chap.~4]{Bu}. That is, the projective distance between points $P$ and $Q$ on a segment $AB$ is as given by the formula
$$
d(P,Q)=\frac{1}{2}\ln \frac{(Q-A)(B-P)}{(P-A)(B-Q)}
$$
(this formula defines distance in the Cayley-Klein, or projective, model of hyperbolic geometry; the factor $1/2$ is needed for the curvature to be $-1$).

\begin{figure}[hbtp]
\centering
\includegraphics[height=.6in]{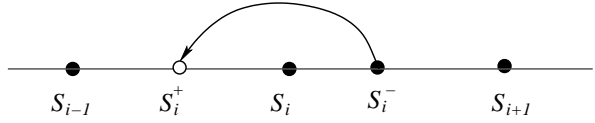}
\caption{Evolution of points in the projective line}
\label{leap}
\end{figure}

\begin{theorem} \label{k2}
 {\rm (i)} The map $\Phi$ is given by the following equivalent equations:
\begin{equation} \label{Men1}
\frac{1}{S^{+}_i-S_i}+\frac{1}{S^{-}_i-S_i}=\frac{1}{S_{i+1}-S_i}+\frac{1}{S_{i-1}-S_i},
\end{equation}
\begin{equation} \label{Men2}
\frac{(S^{+}_i-S_{i+1})(S_i-S^{-}_i)(S_i-S_{i-1})}{(S^{+}_i-S_i)(S_{i+1}-S_i)(S^{-}_i-S_{i-1})}=-1,
\end{equation}
\begin{equation} \label{Men3}
\frac{(S^{+}_i-S_{i-1})(S_i-S^{-}_i)(S_{i+1}-S_{i})}{(S^{+}_i-S_i)(S_{i}-S_{i-1})(S^{-}_i-S_{i+1})}=-1.
\end{equation}
{\rm (ii)} The map induced by $\Phi$ on the moduli space ${\mathcal S}_n/PGL(2,\R)$ is the map $T_2$  given in \eqref{mapxy}.
\end{theorem}

\begin{proof} 
 {\rm (i)} A fractional-linear involution $x\mapsto y$ with a fixed point $a$ is given by the formula
 $$
 \frac{1}{x-a}+\frac{1}{y-a}=b
 $$
 where $b$ is some constant.
 Since the leapfrog involution has $S_i$ as a fixed point and swaps $S_i^-$ with $S_i^+$ and $S_{i-1}$ with $S_{i+1}$, one has
 $$
 \frac{1}{S_i^--S_i}+\frac{1}{S_i^+-S_i}=b= \frac{1}{S_{i-1}-S_i}+\frac{1}{S_{i+1}-S_i},
 $$
 which implies~\eqref{Men1}. That equalities~\eqref{Men2} and~\eqref{Men3} are equivalent to~\eqref{Men1} is verified by a straightforward computation. 
 
{\rm (ii)} We need to check that 
\begin{equation} \label{newold}
x_i^*=x_{i-1} \frac{x_i+y_i}{x_{i-1}+y_{i-1}}, \qquad y_i^*=y_i \frac{x_{i+1}+y_{i+1}}{x_i+y_i},  
\end{equation}
where $(x_i^*,y_i^*)$ are given by formulas~\eqref{xyST} with $S^-$ replaced by $S$ and $S$ by $S^+$. 

The computation is simplified by the observation that
$$
x_i+y_i = \frac{(S_{i+1}^- - S_i^-) (S_{i+2}-S_{i+1})}{(S_{i+1}^- - S_{i+2}) (S_i^- - S_{i+1})}.
$$
After substituting to~\eqref{newold}, a direct computation reveals that the first of the equalities~\eqref{newold} is equivalent to the quotient of~\eqref{Men2} and~\eqref{Men3}, whereas the second is equivalent to their product.
\end{proof}

The leapfrog transformation can be interpreted in terms of hyperbolic geometry. Let us identify $\RP^1$ with the circle at infinity of the hyperbolic plane ${\mathbf H}^2$. Then the restrictions of  hyperbolic isometries on the circle at infinity are the projective transformations of $\RP^1$. Accordingly,  $S^-$ and $S$ are ideal polygons in ${\mathbf H}^2$.

The projective transformation that interchanges the vertices $S_{i-1}$ and $S_{i+1}$ and fixes $S_i$ is the reflection of the hyperbolic plane in the line $L_i$ through $S_i$, perpendicular to the line $S_{i-1}S_{i+1}$ (that is, the altitude of the ideal triangle $S_{i-1} S_i S_{i+1}$); see Fig.~\ref{reflection} where we use the projective (Cayley-Klein) model of the hyperbolic plane. 

\begin{figure}[hbtp]
\centering
\includegraphics[height=1.8in]{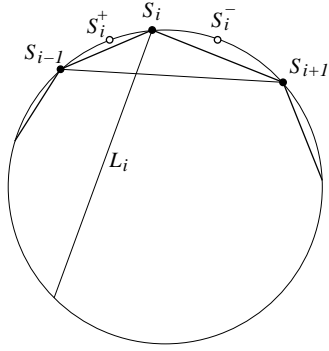}
\caption{Leapfrog transformation in the hyperbolic plane}
\label{reflection}
\end{figure}

The ideal polygon obtained by reflecting each vertex $S^-_i$ in the respective line $L_i$ is $S^+$. Thus the leapfrog transformation $\Phi$ is presented as the composition of two involutions:
$$
(S^-,S) \mapsto (S^+,S) \mapsto (S,S^+).
$$

We note a certain similarity of the map $\Phi$ with the polygon recutting studied by Adler \cite{Ad1,Ad2}, which is also a completely integrable transformation of polygons (in the Euclidean plane or, more generally, Euclidean space). In Adler's case, one reflects the  vertex $V_i$  in the  perpendicular bisector of the diagonal $V_{i-1} V_{i+1}$, after which one proceeds to the next vertex by increasing the index $i$ by 1.

\subsubsection{Lagrangian formulation of leapfrog transformation}
The map $\Phi$ can be described as a discrete Lagrangian system. 
Let us recall  relevant definitions, see, e.g., \cite{Ve}.

Given a manifold $M$, a Lagrangian system is a map $F:M\times M \to M\times M$  defined as follows:
$$
F(x,y)=(y,z)\quad  \text{if and only if} \quad \frac{\partial}{\partial y} \left(L(x,y)+L(y,z)\right) =0,
$$
where $L:M\times M \to \R$ is a function (called the Lagrangian). 

Many familiar discrete time dynamical systems can be described this way (for example, the billiard ball map, for which $L(x,y)=|x-y|$ where $x$ and $y$ are points on the boundary of the billiard table).

Note that the map $F$ does not change if the Lagrangian is changed as follows: 
\begin{equation} \label{change}
L(x,y) \mapsto L(x,y) + g(x)-g(y)
\end{equation}
 where $g$ is an arbitrary function.

A Lagrangian system has an invariant pre-symplectic (that is, closed) differential 2-form
$$
\omega = \sum_{i,j} \frac{\partial^2 L(x,y)}{\partial x_i \partial y_j} dx_i \wedge dy_j.
$$
The form $\omega$ does not change under the transformation~\eqref{change}.

In the next proposition, assume that the $n$-gons in $\RP^1$ under consideration are closed (that is, the monodromy is the identity). As before, we choose an affine coordinate on the projective line and treat the vertices $S_i$ and $S^-_i$ as numbers. The index $i$ is understood cyclically, so that $i=n+1$ is the same as $i=1$.

\begin{proposition} \label{genfunct}
{\rm (i)} The leapfrog map $\Phi$ is a discrete Lagrangian system with the Lagrangian
\begin{equation} \label{Lagr}
L(S^-,S) = \sum_i \ln|S_i - S_{i+1}| - \sum_i \ln|S_i - S^-_i|.
\end{equation}

{\rm (ii)} The Lagrangian $L$ changes under fractional-linear transformation
$$
x \mapsto \frac{ax+b}{cx+d}
$$
as follows:
$L(S^-,S) \mapsto L(S^-,S) + g(S^-) - g(S)$,
where $g(S)=\sum_i \ln |cS_i+d|$.
\end{proposition}

\begin{proof}
(i) Differentiating $L(S^-,S)+L(S,S^+)$ with respect to $S_i$ yields~\eqref{Men1}.

(ii) If $\bar S=(aS+b)/(cS+d)$ then
$$
\bar S_i-\bar S_{i+1}=\frac{D (S_i-S_{i+1})}{(cS_i+d)(cS_{i+1}+d)},\qquad \bar S_i-\bar S^-_{i}=\frac{D (S_i-S^-_{i})}{(cS_i+d)(cS^-_{i}+d)}
$$
with $D=ad-bc$. It follows that 
\begin{multline*}
\sum_i (\ln|\bar S_i - \bar S_{i+1}| -  \ln|\bar S_i - \bar S^-_i|) = 
\sum_i (\ln|S_i - S_{i+1}| -  \ln|S_i - S^-_i)| +\\
\sum_i (\ln |cS^-_i+d| - \ln |cS_i+d)|,
\end{multline*}
as claimed.
\end{proof}

\begin{corollary}
The 2-form 
$$
\omega=\sum_{i=1}^n \frac{dS^{-}_i\wedge dS_i}{(S^{-}_i-S_i)^2}.
$$
is a closed $\PSL(2,\R)$-invariant differential form invariant under the map $\Phi$. 
\end{corollary}

We note that the form $\omega$ is not basic, that is, it does not descend on the quotient space ${\mathcal S}_n/PGL(2,\R)$.

\begin{remark} 
 Numerical simulations show that $\Phi$ does not have integrable behavior.
Figure~\ref{fig:nonintegrableK2} shows chaotic behavior of two
quantities: the horizontal axis is $S_2-S_1$, the vertical axis is
$S_3-S_1$.

\begin{figure}[hbtp]
\centering
\includegraphics[height=2.5in]{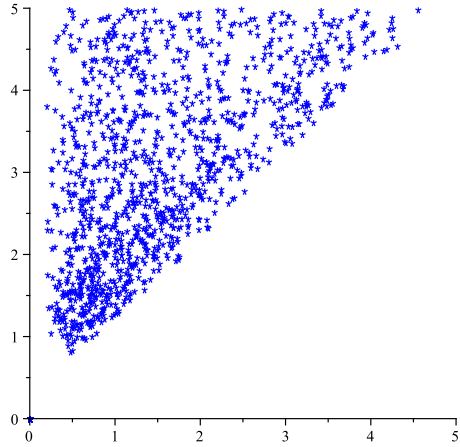}
\caption{Non-integrable behavior of the map $\Phi$ for $n=6$}
\label{fig:nonintegrableK2}
\end{figure}

\end{remark}

\subsubsection{Circle pattern} If the ground field is $\C$ then the mapping $\Phi$ can be interpreted as a certain circle pattern. 

\begin{figure}[hbtp]
\centering
\includegraphics[height=2.5in]{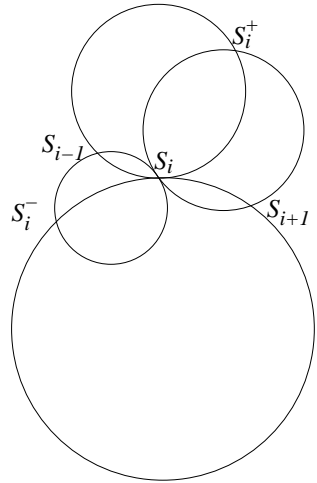}
\caption{Pairwise tangent circles}
\label{circles}
\end{figure}

Consider Fig.~\ref{circles}. This figure depicts a local rule of constructing point $S_i^+$ from points $S_{i-1}, S_i, S_{i+1}$ and $S^{-}_i$. Namely, draw the circle through  points $S_{i-1},S^{-}_i,S_{i}$, and then draw the circle through points $S_i, S_{i+1}$, tangent to the previous circle. Now repeat the construction: draw the circle through points $S_{i+1},S^{-}_i,S_{i}$, and then draw the circle through points $S_i, S_{i-1}$, tangent to the previous circle. Finally, define $S^{+}_i$ to be the intersection point of the two ``new" circles.

\begin{proposition} \label{circint}
The tangency of two pairs of circles meeting at point $S_i$ in Fig.~\ref{circles} is equivalent to 
equations~\eqref{Men1}-\eqref{Men3}.
\end{proposition}

\begin{proof}
A M\"obius transformation sends a circle or a line to a circle or a line. Send point $S_i$ to infinity; then the circles through this points become straight lines. Two circles are tangent if they make zero angle. Since M\"obius transformations are conformal, the respective lines are parallel.

Thus the configuration of circles in Fig.~\ref{circles} becomes a parallelogram with vertices $S_{i-1}, S^-_i, S_{i+1}$ and $S_i^+$. The quadrilateral is a parallelogram if and only if
\begin{equation} \label{para}
S_{i-1}+S_{i+1}=S^-_i+S_i^+.
\end{equation}
On the other hand, if $S_i=\infty$ then equation~\eqref{Men2} becomes
$$
\frac{S^{+}_i-S_{i+1}}{S^{-}_i-S_{i-1}}=-1,
$$
which is equivalent to~\eqref{para}.
\end{proof}

This circle pattern generalizes the one studied by O.~Schramm in~\cite{Sc} in the framework of discretization of the theory of analytic functions (there the pairs of non-tangent neighboring circles were orthogonal). See also \cite{BH} concerning more general  circle patterns with constant intersection angles and their relation with discrete integrable systems of Toda type.

\bigskip

{\bf Acknowledgments}. It is a pleasure to thank the Hausdorff Research Institute for Mathematics whose hospitality the authors enjoyed in summer of 2011 where this research was initiated. We are grateful to A.~Bobenko, V.~Fock, S.~Fomin, M.~Glick, R.~Kedem,
R.~Kenyon, B.~Khesin, G.~Mari-Beffa, V.~Ovsienko, R.~Schwartz, F. Soloviev, Yu.~Suris
for stimulating discussions. M.~G. was partially supported by the
NSF grant DMS-1101462; M.~S. was partially supported by the NSF grants DMS-1101369 and DMS-1362352;
S.~T. was partially supported by the Simons Foundation grant No 209361 and by the NSF grant DMS-1105442;
A.~V. was partially supported by the ISF grant No~162/12.

\end{document}